\documentclass{article}

\usepackage[utf8]{inputenc} 
\usepackage{url}

\usepackage{amsmath,amsfonts,amssymb}
\usepackage{amsthm}
\usepackage{graphicx}
\usepackage{enumerate}
\usepackage{multirow,bigdelim}
\usepackage{mathtools}
\usepackage{scrextend}
\usepackage{soul} 
\usepackage{color}
\usepackage{array}
\usepackage[normalem]{ulem}
\usepackage{url}
\usepackage{multicol}
\usepackage{enumitem}
\usepackage{hyperref}
\usepackage{empheq}
\usepackage{cases}
\usepackage{xcolor}
\hypersetup{
	       colorlinks,
	       linkcolor={red!50!black},
	       citecolor={blue!50!black},
	       urlcolor={blue!80!black}
	}

\definecolor{myred}{RGB}{255,50,50}         
\definecolor{myblack}{RGB}{0,0,0} 

\newcommand{\SOC}[1]{{\mathcal{Q}^{#1}}}

\newcommand{\reInt}{\mathrm{ri}\,}

\newcommand{\lineality}{\mathrm{lin}\,}

\newcommand{\closure}{\mathrm{cl}\,}

\newcommand{\spanVec}{\mathrm{span}\,}
\newcommand{\norm}[1]{\lVert{#1}\rVert}
\newcommand{\inProd}[2]{\langle #1 , #2 \rangle }

\newcommand{\feasS}{\mathcal{F}_{\text{D }}^s}
\newcommand{\PSDcone}[1]{{\mathcal{S}^{#1}_+}}	
\newcommand{\nonNegative}[1]{{\mathcal{N}^{#1}}}	
\newcommand{\doubly}[1]{{\mathcal{D}^{#1}}}

\newcommand{\minFaceD}{ {\mathcal{F}_{\min}^D}}

\newcommand{\minFace}[1]{ {\mathcal{F}_{\min}^{#1}}}

\newcommand{\stdMap}{ {\mathcal{A}}}

\newcommand{\stdCone}{ {\mathcal{K}}}
\newcommand{\stdSpace}{ \mathcal{L}}

\newcommand{\stdAffine}{ \mathcal{V}}
\newcommand{\stdFace}{ \mathcal{F}}

\newcommand{\stdInt}{ {e}}
\newcommand{\pOpt}{ {\theta _P}}
\newcommand{\dOpt}{ {\theta _D}}
\newcommand{\opt}[1]{ {\theta _{#1}}}
	
\newcommand{\matRange}{{\mathrm{ range } \,}}

\renewcommand{\Re}{\mathbb{R}}    
\newcommand{\fDom}{\mathrm{ dom } \,}  
\renewcommand{\S}{\mathcal{S}}                    

\newcommand{\distP}{\ell _{\text{poly}}}        
\newcommand{\distS}{d _{\text{str}}} 
\newcommand{\dimSpace}{{\mathrm{dim}} \,}  
\newcommand{\PPS}{{PPS}}
\newcommand{\T}{\ast\hspace{-1pt}} 
\newtheorem*{remark}{Remark}

\title{Facial Reduction and Partial Polyhedrality}

\author{Bruno F. Louren\c{c}o
\thanks{Department of Mathematical Informatics, Graduate School of Information Science \& Technology, University of Tokyo, 7-3-1 Hongo, Bunkyo-ku, Tokyo 113-8656, Japan.
	(Email: \mbox{lourenco@mist.i.u-tokyo.ac.jp})}
        \and
        Masakazu Muramatsu\thanks{
                     Department of Computer and Network Engineering, The University of Electro-Communications 1-5-1 Chofugaoka, Chofu-shi, Tokyo, 182-8585 Japan. (E-mail: \mbox{MasakazuMuramatsu@uec.ac.jp})
                  }
         \and       
                Takashi Tsuchiya
\thanks{
National Graduate Institute for Policy Studies 7-22-1 Roppongi, Minato-ku, Tokyo 106-8677, Japan. (E-mail: \mbox{tsuchiya@grips.ac.jp}) \newline M. Muramatsu and T. Tsuchiya are supported in part with Grant-in-Aid for Scientific Research (B)24310112 and 
(C) 26330025. M. Muramatsu is  also partially supported by the
Grant-in-Aid for Scientific Research (B)26280005. T. Tsuchiya is also partially supported by the Grant-in-Aid for Scientific Research (B)15H02968.
                  }
        }

\date{November 2015 (Revised: October 2016, April 2017, April 2018)}

\newtheorem{definition}{Definition}
\newtheorem{lemma}[definition]{Lemma}
\newtheorem{proposition}[definition]{Proposition}
\newtheorem{example}{Example}

\newtheorem{corollary}[definition]{Corollary}

\newtheorem{theorem}[definition]{Theorem}

\begin{document}
\maketitle
\begin{abstract}
We present FRA-Poly, a facial reduction algorithm (FRA) for conic linear programs 
that is sensitive to the presence of polyhedral faces in the cone. 
The main goals of FRA and FRA-Poly are the same, i.e., 
finding the minimal face containing the feasible region and detecting
infeasibility, but FRA-Poly treats polyhedral constraints separately.  
This reduces the number of iterations drastically when there are many linear 
inequality constraints.
The worst case number of iterations for FRA-Poly is written in the terms of a 
``distance to polyhedrality'' quantity and 
provides better bounds than FRA under mild conditions. In particular, in the case 
of the doubly nonnegative cone, FRA-Poly gives a worst case bound of $n$ whereas the classical FRA is $\mathcal{O}(n^2)$. 
Of possible independent interest, we  prove a variant 
of Gordan-Stiemke's Theorem and a proper separation theorem that takes into account partial polyhedrality. We provide a discussion on the optimal facial reduction 
strategy and an instance that forces FRAs to perform many steps. 
We also present a few applications. In particular, we will use 
FRA-Poly to improve the bounds recently obtained 
by Liu and Pataki on the dimension of certain affine subspaces which appear in 
weakly infeasible problems.  
\end{abstract}

\section{Introduction}
Consider the following pair of primal and dual conic linear programs (CLPs):
\vspace*{-3\baselineskip} 
\begin{multicols}{2}
\begin{align}
\underset{x}{\inf} & \quad \inProd{c}{x} \label{eq:primal}\tag{P}\\ 
\mbox{subject to} & \quad \stdMap x = b \nonumber \\ 
&\quad x \in \stdCone ^* \nonumber 
\end{align}
	\break
	\begin{align}
	\underset{y}{\sup} & \quad \inProd{b}{y} \label{eq:dual} \tag{D} \\ 
	\mbox{subject to} & \quad c - \stdMap ^\T y \in \stdCone \nonumber,\\ \nonumber
	\end{align}	
\end{multicols}
\noindent where $\stdCone \subseteq \Re^n$ is a closed convex cone and $\stdCone^*$ is the dual cone $\{s \in \Re^n \mid \inProd{s}{x} \geq 0, \forall x \in \stdCone \}$. We have 
that $\stdMap: \Re^n \to \Re^m$ is a linear map, $b \in \Re^m$, $c \in \Re^n$ and 
$\stdMap^\T$ denotes the adjoint map. We also have $\stdMap^\T y = \sum _{i=1}^m \stdMap _i y_i$, for certain 
elements $\stdMap _i \in \Re^n$. The  inner product is denoted by $\inProd{\cdot}{\cdot}$.
We will use $\pOpt$ and $\dOpt$ to denote the primal and  dual optimal value, respectively.
It is understood that $\pOpt = +\infty$ if \eqref{eq:primal} is infeasible and 
$\dOpt = -\infty$ if \eqref{eq:dual} is infeasible.

In the absence of either a primal relative interior feasible solution or a dual 
relative interior slack, it is possible that $\pOpt \neq \dOpt$. A possible 
way of  correcting that is to let $\minFaceD$ be 
the minimal face of $\stdCone$ which contains the feasible slacks $\feasS = \{c- \stdMap ^\T y \in \stdCone \mid y \in \Re^m  \}$, then 
we substitute $\stdCone$ by $\minFaceD$ and $\stdCone^*$ by $(\minFaceD)^*$. With that, 
the new primal optimal value $\pOpt ' $ will satisfy $\pOpt ' = \dOpt$. This is precisely what 
facial reduction \cite{Borwein1981495, pataki_strong_2013, article_waki_muramatsu} approaches do.

In this paper, we analyze how to take advantage of the presence of polyhedral faces in 
$\stdCone$ when doing Facial Reduction. To do that, 
we introduce  FRA-Poly, which is a facial reduction algorithm (FRA) that, in many cases, provides a better worst case complexity than 
the usual approach, especially when $\stdCone$ is a direct product of several 
cones. The idea behind it is as follows. 
Facial reduction algorithms work by successively identifying what is called ``reducing 
directions'' $\{d_1, \ldots, d_\ell \}$. Starting with $\stdFace _1 = \stdCone$, these 
directions define faces of $\stdCone$ by the relation $\stdFace _{i+1} = \stdFace _{i} \cap \{d_i\}^\perp$. For feasible problems, $d_i$ must be 
such that  $\stdFace _{i+1}$ is a face of $\stdCone$ containing $\feasS$.
We then obtain a sequence  $\stdFace _1 \supseteq \ldots \supseteq \stdFace _\ell$ of faces of 
$\stdCone$ such that $\stdFace _i \supseteq \feasS$ for every $i$. Usually, a FRA proceeds 
until  $\minFaceD$ is found.

A key observation is that as soon as we reach a polyhedral face $\stdFace _i$, we can 
jump to the minimal face $\minFaceD$ in a single facial reduction step. 
In addition, when $\stdCone$ is a direct product $\stdCone = \stdCone^1 \times \ldots \times \stdCone^r$, 
each $\stdFace _i$ is also a direct product  $\stdFace_i^1 \times \ldots \times \stdFace_i^r$.
In this case an even weaker condition is sufficient to jump to $\minFaceD$, namely, 
if every block $\stdFace _i ^j$ is polyhedral or 
it is already equal to $j$-th block of the minimal face.

Our proposed algorithm FRA-Poly works in two phases. In Phase 1, it proceeds until a face 
$\stdFace _i$ satisfying  the  condition above is reached or until a certificate of 
infeasibility is found. In Phase 2, $\minFaceD$ 
is obtained with single facial reduction step. 
One  interesting point is 
that even if $\stdFace _i \neq \minFaceD$, if we substitute 
$\stdCone$ for $\stdFace _i$ in \eqref{eq:dual}, then strong 
duality will hold. The theoretical backing for that is given by Proposition \ref{prop:slater2}, which 
is a generalization of the classical strong duality theorem. In Section \ref{sec:partial}, 
we will give a generalization of the Gordan-Stiemke 
Theorem for the case when $\stdCone$ is the direct product of a closed convex cone and a polyhedral cone, see 
Theorem \ref{theo:partial_Gordan-Stiemke}. We also prove a proper separation theorem that will be 
the engine behind FRA-Poly, see Theorem \ref{theo:reint_partial}.

In order to analyze the number of facial reduction steps, 
we introduce a quantity called \emph{distance to polyhedrality} $\distP(\stdCone)$. This is the 
length \emph{minus one} of the longest strictly ascending chain of nonempty faces 
$\stdFace _1 \subsetneq \ldots \subsetneq \stdFace _\ell$ for 
which $\stdFace _1$ is polyhedral and $\stdFace _i$ is not polyhedral for all $i > 1$.  
If $\stdCone$ is a direct product of  cones $\stdCone^1 \times \ldots \times \stdCone^r$, 
we prove that FRA-Poly stops in at most $1 + \sum _{i = 1}^r \distP(\stdCone^i)$ steps. This is no 
worse than the bound given by classical FRA and, provided that at least two of the cones 
$\stdCone^i$ are not subspaces, it is strictly smaller. We also discuss whether our bounds 
are achieved by some problem instance, see Section \ref{sec:tight} and Proposition \ref{prop:worst_bound} (Appendix \ref{app:ex}).

As an application, we give a nontrivial bound for the singularity degree of CLPs
over cones that are intersections of two other cones. In particular, 
for the case of the doubly nonnegative cone $\doubly{n}$, we show 
that the longest chain of nonempty faces of $\doubly{n}$ has length $1 + \frac{n(n+1)}{2}$.
Therefore, the classical analysis gives the upper bound    $\frac{n(n+1)}{2}$ for the singularity
 degree of feasible problems over $\doubly{n}$.
On the other hand, using our technique, we show that the singularity degree of any problem over 
$\doubly{n}$ is at most $n$. We also use FRA-Poly to improve  bounds obtained by Liu and Pataki in Corollary 1 of \cite{LP17} on 
the dimension of certain subspaces connected to weakly infeasible problems.

Table \ref{table:bounds_summary} contains a summary of the bounds 
predicted by FRA and FRA-Poly for several cases. The notation 
$\ell _{\stdCone}$ indicates the length of the longest strictly ascending  chain of 
nonempty faces of $\stdCone$.
The first line corresponds 
to a single cone, the second to a product of $r$ arbitrary closed convex cones 
and the third to the product of $r_1$ Lorentz  cones and $r_2$ positive 
semidefinite cones, respectively. These results follow from Theorem \ref{theo:fra_bound} and 
 Example \ref{ex:psd_lorentz}. 
The last line contains the bounds for the doubly nonnegative cone, which follows 
from Proposition \ref{prop:doubly_chain} and  Corollary \ref{col:doubly_sing}.

\begin{table}
\centering
\begin{tabular}{c|c|c}
 & FRA & FRA-Poly \\  \hline
$\stdCone$ & $\ell_\stdCone $  & $1 + \distP(\stdCone)$\\
$\stdCone^1\times \ldots \times \stdCone^r$ & 1 + $\sum_{i=1}^{r} (\ell_{\stdCone^i} -1)$   &  1 + $\sum_{i=1}^{r} \distP({\stdCone^i})$ \\ 
$\SOC{t_1} \times \ldots \times \SOC{t_{r_1}} \times \PSDcone{n_1}  \times \ldots \times \PSDcone{n_{r_2}}$ & $1+ 2r_1 + \sum _{j = 1}^{r_2} n_j$ & $1 + r_1 + \sum _{j = 1}^{r_2} (n_j  - 1)$  \\ 
$\doubly{n}$ & $1 + \frac{n(n+1)}{2}$  & $n$ \\ 
\end{tabular} \caption{Summary of the worst case number of reduction steps predicted by the classical 
FRA analysis and by FRA-Poly.}\label{table:bounds_summary}
\end{table}


This work is divided as follows. In Section \ref{sec:not} we give some background on 
 related notions. In 
 Section \ref{sec:fra}, we review facial reduction.  In Section \ref{sec:partial} we prove 
 versions of two classical theorems taking into account partial polyhedrality.
 In Section \ref{sec:fra_poly} we analyze FRA-Poly and in Section \ref{sec:app} we 
 discuss two applications. Appendix \ref{app:partial} contains the proof of a strong duality 
 criterion. Appendix \ref{app:ex} illustrates FRA-Poly and contains an example which generalizes 
 an earlier worst case SDP instance by Tun\c{c}el.
 

\section{Notation and Preliminaries}\label{sec:not}
In this section, we will define the notation used throughout this article and 
review a few concepts. More details  can be found in \cite{rockafellar,pataki_handbook}.
For $C$ a closed convex set, we will denote by $\reInt C$ and 
$\closure C$ the relative interior and the closure of $C$, respectively.
If $\mathcal{U}$ is an 
arbitrary set, we denote by $\mathcal{U}^\perp$ the subspace which contains 
the elements orthogonal to it. We will denote by $\PSDcone{n}$ the cone of 
$n\times n$ symmetric positive semidefinite matrices and by $\SOC{n}$ the 
Lorentz cone $\{(x_0,\overline{x}) \in \Re\times \Re^{n-1} \mid x_0 \geq \norm{\overline{x}}_2 \}$, 
where $\norm{.}_2$ is the usual Euclidean norm.
The nonnegative orthant will be denoted by $\Re^n_+$.

If $\stdCone$ is a closed convex cone, we write $\stdCone^*$ for its dual 
cone. We write $\spanVec \stdCone$ for its linear span and $\lineality \stdCone$ for 
its lineality space, which is $\stdCone \cap -\stdCone$. $\stdCone$ is said to be \emph{pointed} if 
$\lineality \stdCone  = \{0\}$. We have $\lineality (\stdCone^*) = \stdCone^\perp$, see Theorem 14.6 in \cite{rockafellar}. Also, if we select $\stdInt \in \reInt \stdCone$ and $x \in \stdCone^*$, then $x \in \stdCone^\perp $ if and only if $\inProd{\stdInt}{x} = 0$\footnote{Suppose $\inProd{\stdInt}{x} = 0$. Then, given $y \in \stdCone$ we have $\alpha \stdInt +(1-\alpha)y \in \stdCone$ for some $\alpha > 1$, due to Theorem 6.4 in \cite{rockafellar}. Taking the inner product with $x$, we see that $\inProd{y}{x}$ must be zero.}.

The conic linear program \eqref{eq:dual} can be in four different feasibility 
statuses: $i)$ strongly feasible if $\feasS \cap \reInt \stdCone \neq \emptyset$; $ii)$ weakly feasible if 
$\feasS \neq \emptyset$ but $\feasS \cap \reInt \stdCone  = \emptyset$; $iii)$ weakly infeasible if 
$\feasS = \emptyset$ but $\text{dist}(c + \matRange \stdMap^\T ,\stdCone) = 0$; $iv)$ strongly 
infeasible if $\text{dist}(c + \matRange \stdMap^\T ,\stdCone) > 0$.  Note that 
\eqref{eq:primal} admits analogous definitions. 

The 
strong duality theorem states that if \eqref{eq:dual} is strongly feasible and 
$\dOpt < + \infty$, then $\pOpt = \dOpt$ and $\pOpt$ is attained. 
On the other hand, if \eqref{eq:primal} is strongly feasible and 
$\pOpt > -\infty$,  then $\pOpt = \dOpt$ and $\dOpt$ is attained. 
We  will also need a special version of the strong duality theorem. 
First, we need the following definition.

\begin{definition}[Partial Polyhedral Slater's condition]\label{def:pps}
	Let $\stdCone = \stdCone ^1\times \stdCone ^2$, where  $\stdCone ^1\subseteq \Re^{n_1},\stdCone ^2 \subseteq \Re^{n_2}$ are closed convex cones such that $\stdCone ^2$ is polyhedral.
	We say that \eqref{eq:dual} satisfies the Partial Polyhedral Slater's (\PPS) condition if there is 
	a slack $(s_1,s_2) = c-\stdMap^\T y$, such that $s_1 \in \reInt \stdCone^1$ and $s_2 \in \stdCone^2$.
	Similarly, we say that \eqref{eq:primal} satisfies the {\PPS} condition, if there is a primal 
	feasible solution $x = (x_1,x_2)$ for which $x_1 \in \reInt (\stdCone^1)^*$.
\end{definition}
The following is a strong duality theorem based on the {\PPS} condition.
As we could not find a precise reference for it, we give a proof in the Appendix \ref{app:partial}.  
\begin{proposition}[{\PPS}-Strong Duality]\label{prop:slater2}
	Let $\stdCone = \stdCone ^1\times \stdCone ^2$, where  $\stdCone ^1\subseteq \Re^{n_1},\stdCone ^2 \subseteq \Re^{n_2}$ are closed convex cones such that $\stdCone ^2$ is polyhedral.
	\begin{enumerate}[label=(\roman*)]
		\item If $\pOpt$ is finite and \eqref{eq:primal} satisfies the {\PPS} condition, then $\pOpt = \dOpt$ and the dual optimal value is attained.
		\item If $\dOpt$ is finite and \eqref{eq:dual} satisfies the {\PPS} condition, then $\pOpt = \dOpt$ and the primal 
		optimal value is attained.
	\end{enumerate}
\end{proposition}

\section{Facial Reduction}\label{sec:fra}
Facial Reduction was developed by Borwein and Wolkowicz to restore strong duality 
in convex optimization \cite{borwein_facial_1981,Borwein1981495}. Descriptions for 
the conic linear programming case have appeared, for instance, in Pataki \cite{pataki_strong_2013} and 
in Waki and Muramatsu \cite{article_waki_muramatsu}.

Here, we will suppose that our main interest is in the dual problem \eqref{eq:dual}.
 Facial Reduction hinges on the fact that 
strong feasibility fails  if and only if there 
is $d \in \stdCone ^*$ such that $\stdMap d = 0$  and
one of the two alternatives holds: $(i)$  $\inProd{c}{d} = 0$ and $d \not \in \stdCone ^\perp$; or 
$(ii)$ $\inProd{d}{c} < 0$, see Lemma 3.2 in \cite{article_waki_muramatsu}. If 
alternative $(i)$ holds, $\stdFace = \stdCone \cap \{d\}^\perp$ is a proper face of $\stdCone$ 
containing $\feasS$.  We then substitute  $\stdCone$ by $\stdFace $ and repeat.
If  $(ii)$ holds, \eqref{eq:dual} is infeasible. 
 We write below a generic 
facial reduction algorithm similar to the one described in \cite{article_waki_muramatsu}.

\noindent
[{\bf Generic Facial Reduction}]\label{alg:gen_fra}

{\bf Input:} \eqref{eq:dual}

{\bf Output:} A set of reducing directions $\{d_1, \ldots, d_\ell\}$ and $\minFaceD$.
\begin{enumerate}
\item $\stdFace _1 \leftarrow \stdCone, i \leftarrow 1 $
\item Let $d_i$ be an element in $\stdFace _i ^* \cap \ker \stdMap$ such that either: 
$i)$ $d_i \not \in \stdFace _i^\perp$ and $\inProd{c}{d_i} = 0$; or $ii)$ $\inProd{c}{d_i} < 0$.
If no such $d_i$ exists, let $\minFaceD \leftarrow \stdFace _i$ and stop. 
\item If $\inProd{c}{d_i} < 0$, let $\minFaceD \leftarrow  \emptyset$ and stop.
\item If $\inProd{c}{d_i} = 0$, let $\stdFace _{i+1}  \leftarrow \stdFace _i  \cap \{d_i\}^\perp, i\leftarrow i+1$ and return to 2).
\end{enumerate}

We will refer to the directions satisfying $d_i \in \stdFace _i^* \cap \ker \stdMap$ and 
$\inProd{c}{d_i} \leq 0$ as \emph{reducing directions}, so that the $d_i$ in Step 2. are indeed reducing directions. An important issue when doing facial reduction is how to model the 
search for the reducing directions. It is sometimes said that doing facial 
reduction can be as hard as solving the original problem. However, an important 
difference is that the search for the $d _i$ can be cast as a pair of primal and 
dual problems which are always strongly feasible.
This was shown in the work by Cheung, Schurr and Wolkowicz \cite{csw13} and 
in our previous work \cite{lourenco_muramatsu_tsuchiya3}. 
Recently,
Permenter, Friberg and Andersen showed that $d_i$ can also be obtained as by-products 
of self-dual homogeneous methods \cite{PFA15}.
There are 
also approximate approaches such as the one described by Permenter and Parrilo \cite{PP14}, 
where the search for the $d_i$ is conducted in a more tractable cone at 
the cost of, perhaps, failing to identify $\minFaceD$, but still simplifying 
the problem nonetheless. See also the article by Frieberg \cite{Fr16}, where 
conic constraints are dropped when searching for the reducing directions, making it easier 
to find the $d_i$ but introducing  representational issues.

In this article, we will search for reducing directions by
considering the pair \eqref{eq:red} and \eqref{eq:red_dual} introduced in \cite{lourenco_muramatsu_tsuchiya3}, which 
are parametrized by $\stdMap$, $c$, $\stdCone, \stdInt, \stdInt^*$. In Phase 1 of FRA-Poly, we will always 
select $\stdInt$ and $\stdInt^*$ according to Lemma \ref{lemma:red2}. Different choices will be discussed/used 
in Phase 2 of FRA-Poly and on Sections \ref{sec:compl} and \ref{sec:tight}.
\begin{align}
\underset{x,t,w}{\mbox{minimize}} & \quad t & \tag{$P_\stdCone $}\label{eq:red}  \\ 
\mbox{subject to} & \quad -\inProd{c}{x  -t\stdInt^*} + t - w &= 0 \label{eq:red:1}\\
&\quad \inProd{\stdInt}{x} + w &= 1 \label{eq:red:2}\\
& \quad\stdMap x  -t\stdMap \stdInt^* & = 0 \label{eq:red:3}\\
&\quad (x,t,w) \in \stdCone ^* \times \Re_+ \times \Re_+ \nonumber\\
\underset{y_1,y_2,y_3}{\mbox{maximize}} &\quad  y_2 & \tag{$D_\stdCone $}\label{eq:red_dual}  \\ 
\mbox{subject to} & \quad cy_1 -\stdInt y_2 -\stdMap^\T y_3 \in \stdCone \label{eq:red_dual:1} \\
&\quad 1-y_1(1+\inProd{c}{\stdInt^*}) + \inProd{\stdInt^* }{\stdMap ^\T y_3}\geq 0 \label{eq:red_dual:2} \\
& \quad y_1-y_2 \geq 0 \label{eq:red_dual:3} \\
&\quad (y_1,y_2,y_3) \in \Re \times \Re \times \Re^m. \nonumber
\end{align}

Recall that our goal is to find a point $x \in \ker \stdMap \cap \stdCone$ satisfying  $\inProd{c}{x} \leq 0$. The idea behind  \eqref{eq:red} is to shift the problem by $-t\stdInt^*$ (Equations \eqref{eq:red:1} and \eqref{eq:red:3}) and add constraints to ensure 
the $x$ stays in a bounded region (Equation \eqref{eq:red:2}). These changes ensure that \eqref{eq:red} and \eqref{eq:red_dual} 
satisfy the PPS condition when the parameters $\stdInt, \stdInt^*$ are chosen appropriately as in
the next section.

\section{Partial Polyhedrality Theorems}\label{sec:partial}

We are now in position to present a choice of $\stdInt, \stdInt^*$ for \eqref{eq:red} and \eqref{eq:red_dual}
taking into account the {\PPS} condition.

\begin{lemma}\label{lemma:red2}
Let $\stdCone = \stdCone ^1 \times \stdCone ^2$ be a 
closed convex cone, such that $\stdCone ^2$ is polyhedral.
Consider the pair \eqref{eq:red} and \eqref{eq:red_dual} with 
$\stdInt$ and $\stdInt^*$ such that  $\stdInt = (\stdInt_1,0) \in (\reInt \stdCone ^1)\times \{0\}$ and 
$\stdInt^* \in \reInt \stdCone^*$.  The following properties hold.
\begin{enumerate}[label=(\roman*)]
\item The following are solutions to \eqref{eq:red} and \eqref{eq:red_dual} that satisfy the {\PPS} condition: 
\begin{align*}
(x,t,w)& = \left(\frac{\stdInt^*}{\inProd{\stdInt}{\stdInt^*}+1},\frac{1}{\inProd{\stdInt}{\stdInt^*}+1}, \frac{1}{\inProd{\stdInt}{\stdInt^*}+1}\right) \\
 (y_1,y_2,y_3) & = (0,-1,0).
 \end{align*} 
In particular, 
$\opt{\text{\ref{eq:red}}} =  \opt{\text{\ref{eq:red_dual}}}$.
\end{enumerate}
Let $(x^*,t^*,w^*)$ be any primal optimal solution and 
$(y_1^*,y_2^*,y_3^*)$ be any dual optimal solution.
 \begin{enumerate}[label=(\roman*)]
 \setcounter{enumi}{1} 
\item $\opt{\text{\ref{eq:red}}} =  \opt{\text{\ref{eq:red_dual}}} = 0$ if and only if 
one of the two alternatives below holds:
\begin{enumerate}
\item $\inProd{c}{x^*} < 0 $ and $\feasS = \minFaceD = \emptyset$, or
\item $\inProd{c}{x^*} = 0 $, $\feasS \subseteq  \stdCone \cap \{x^*\}^\perp \subsetneq \stdCone$ and 
$x^*_1 \not \in (\stdCone^1)^\perp = \lineality( (\stdCone^1 )^*)$.
\end{enumerate}

\item $\opt{\text{\ref{eq:red}}} =  \opt{\text{\ref{eq:red_dual}}} >  0$ if and only if the {\PPS} condition is satisfied for \eqref{eq:dual}. In this case, we have
$
c  -\stdMap^\T \frac{y_3^*}{y_1^*} \in (\reInt \stdCone ^1)\times \{0\}. 
$ (In particular, it is feasible for \eqref{eq:dual})

\end{enumerate}
\end{lemma}

%
%
%
%
%
%
\begin{proof}

\begin{enumerate}[label=({\it{\roman*}})]
\item  Due to choice of $\stdInt$ and $\stdInt^*$, it is clear that the solutions 
meet the PPS condition.

\item \fbox{$\opt{\text{\ref{eq:red}}} = 0 \Rightarrow$  (a) or (b) holds.}
Suppose that $\opt{\text{\ref{eq:red}}} = 0$ and let $(x^*,0,w^*)$ be an optimal solution for \eqref{eq:red}.
Due to the constraints in \eqref{eq:red}, we  have
$$\stdMap x^* = 0, \qquad -\inProd{c}{x^*} = w^* \geq 0, \qquad x^* \in \stdCone ^*.$$ 
Note that if $\inProd{c}{x^*} < 0$, then $\minFaceD = \feasS = \emptyset$. This is 
alternative (a).

If $\inProd{c}{x^*} = 0$, then (\eqref{eq:red:1}) implies 
$w^* = 0$. Since $\stdMap x^* = 0$, we obtain  
$$\feasS \subseteq  \stdCone \cap \{x^*\}^\perp. $$
By \eqref{eq:red:2}, we have $\inProd{\stdInt}{x^*} = 1$, so that $x^* = (x_1^*,x_2^*)$ satisfies $x_1^* \not \in (\stdCone^1)^\perp$, due to the choice of $\stdInt$. 
Therefore, the inclusion $\stdCone \cap \{x^*\}^\perp \subseteq \stdCone$ is strict, 
that is,  $\stdCone \cap \{x^*\}^\perp \subsetneq \stdCone$. This is alternative (b).

\fbox{(a) or (b) holds $\Rightarrow \opt{\text{\ref{eq:red}}} = 0$.} 
First recall that  $t$ is constrained to be nonnegative, therefore  $\opt{\text{\ref{eq:red}}} \geq 0$.
Next, we will prove the following inequality
\begin{equation}\label{eq:lemma:aux}
	{\inProd{\stdInt}{x^*} - \inProd{c}{x^*}} > 0.
\end{equation}
Since $\stdInt \in \stdCone,x^* \in \stdCone^*$, we have 
$\inProd{\stdInt}{x^*} \geq 0$. If $(a)$ holds, then $- \inProd{c}{x^*}> 0$.
In particular, \eqref{eq:lemma:aux} holds.
Now, we consider the case where (b) holds. 
In this case, we have $\inProd{c}{x^*} = 0$. 
If $\inProd{\stdInt}{x^*}$ is also zero, then $\inProd{\stdInt_1}{x_1^*}$ is zero and since $\stdInt_1 \in \reInt \stdCone^1$, $ x^*_1$ must 
belong to $(\stdCone^1)^\perp$. However, this is not allowed 
since we assumed that $x_1^* \not \in (\stdCone^1)^\perp$.
Therefore, we have $\inProd{\stdInt}{x^*} > 0$ and, in particular,  \eqref{eq:lemma:aux} also holds.

Let  $$
\alpha = \frac{1}{\inProd{\stdInt}{x^*} - \inProd{c}{x^*}}.
$$
Then $\alpha > 0$ and 
$(\alpha x^*,0,-\alpha \inProd{c}{x^*})$ is a solution to \eqref{eq:red} with value $0$, 
so that $\opt{\text{\ref{eq:red}}} = 0$.

\item 
\fbox{ $\opt{\text{\ref{eq:red_dual}}} > 0 \Rightarrow $ PPS holds for \eqref{eq:dual}}
From \eqref{eq:red_dual:1}, we have
$$
c y_1^* - \stdInt y_2^* - {\stdMap}^\T y_3^* \in \stdCone.
$$
Since $y_2^*=\opt{\text{\ref{eq:red_dual}}}
>0$, we have, due to \eqref{eq:red_dual:3}, $y_1^* > 0$.
 Hence, 
 $$c - \stdInt \frac{y_2^*}{y_1^*} - {\stdMap}^\T \frac{y_3^*}{y_1^*} \in \stdCone.$$ 
We conclude that 
$c - {\stdMap}^\T y_3^*/y_1^* \in (\stdCone^1 +y_2^*/y_1^* \stdInt_1)\times  \{0\} \subseteq (\reInt \stdCone ^1)\times \{0\}$, due to the choice of  $\stdInt$ and the fact that $y_2^*/y_1^* > 0$. 

\fbox{ $\opt{\text{\ref{eq:red_dual}}} = 0 \Rightarrow $ PPS does not hold for \eqref{eq:dual}  } 
If $t^* = 0$ and $(x^*,0,w^*)$ is an optimal solution 
for \eqref{eq:red}, then either $(a)$ or $(b)$ of item $(ii)$ is satisfied. 
If $(a)$ is satisfied, then \eqref{eq:dual} is infeasible and thus the PPS condition
cannot hold. 
If $(b)$ is satisfied, then $\inProd{c}{x^*} = 0$,
$\stdMap x^* = 0$ and $x_1^* \not \in (\stdCone ^1)^\perp$. If 
$(s_1,s_2)$ is a feasible slack for \eqref{eq:dual}, we have $\inProd{s_1}{x_1^*} + \inProd{s_2}{x_2^*} = 0$, so that 
$\inProd{s_1}{x_1^*} = 0$.  
As $x_1^* \not \in (\stdCone ^1)^\perp$, we have that $s_1 \not \in \reInt \stdCone^1$. This means
that there is no feasible solution for \eqref{eq:dual} satisfying the PPS
condition.
\end{enumerate}
\end{proof}

We now prove a theorem that dualizes the criterion in Proposition \ref{prop:slater2}.

\begin{theorem}\label{theo:reint_partial}
	Let $c \in \Re^n$, $\stdSpace \subseteq \Re^n$ be a subspace and $\stdCone = \stdCone ^1 \times \stdCone ^2$ be a 
	closed convex cone, such that $\stdCone ^2$ is polyhedral. Then 
	$(\stdSpace+c)\cap ( (\reInt \stdCone ^1) \times \stdCone^2 ) = \emptyset$ if and only if 
	one of the conditions below holds:
	\begin{enumerate}[label=({\alph*})]
		\item there exists $x \in \stdCone^* \cap \stdSpace^\perp$ such that 
		$\inProd{c}{x} < 0$;
		\item  there exists $x = (x_1,x_2) \in \stdCone^* \cap \stdSpace^\perp \cap \{c\}^\perp$ such that 
		$x_1 \not \in (\stdCone ^1)^\perp$.
	\end{enumerate}
\end{theorem}
\begin{proof}
	Select a linear map $\stdMap$ such that $\stdSpace = \matRange \stdMap^\T $ (therefore, $\stdSpace^\perp = \ker \stdMap$) 
	and consider the problem \eqref{eq:dual} and the pair of problems \eqref{eq:red} and \eqref{eq:red_dual}. 
	Note that $(\stdSpace+c)\cap ( (\reInt \stdCone ^1) \times \stdCone^2 ) = \emptyset$ if and 
	only if the {\PPS} condition is \emph{not} satisfied for \eqref{eq:dual}.
	The result then follows from Lemma \ref{lemma:red2}.	
\end{proof}

One of the points of doing facial reduction is to solve problems that  would not be 
solvable directly if, say, we fed them to an interior point method based solver. 
Therefore, it is natural to consider whether the problems \eqref{eq:red}, 
\eqref{eq:red_dual} are themselves solvable, with the choice of $\stdInt,\stdInt^*$ provided 
by Lemma \ref{lemma:red2}.
We note that due to item $(i)$ of Lemma \ref{lemma:red2}, 
the pair \eqref{eq:red}, \eqref{eq:red_dual} can be solved by infeasible interior-point 
methods in the case of semidefinite and second order cone programming, even though 
they might fail to be strongly feasible.
This is because the convergence theory relies on the existence of optimal solutions 
affording zero duality gap, rather than strong feasibility. See, for instance, 
item 2. of Theorem 11 in the work by Nesterov, Todd and Ye \cite{nesterov_infeasible}.

We remark that Theorem \ref{theo:reint_partial} implies a version of the Gordan-Stiemke's Theorem that takes into account partial polyhedrality.  It contains as a special case the classical version described in Corollary 2 in
Luo, Sturm and Zhang \cite{Luo97dualityresults}. 

\begin{corollary}[Partial Polyhedral Gordan-Stiemke's Theorem]\label{theo:partial_Gordan-Stiemke}\label{col:partial_gordan}
Let $\stdSpace$ be a subspace and $\stdCone = \stdCone ^1 \times \stdCone ^2$ be a 
closed convex cone, such that $\stdCone ^2$ is polyhedral. Then:
\begin{equation*}
\stdCone^* \cap \stdSpace^{\perp} \subseteq (\lineality(( \stdCone ^1)^*)) \times (\stdCone ^2)^* \Leftrightarrow  \left( (\reInt ({\stdCone^1}))\times (\stdCone^2) \right)\cap \stdSpace  \neq \emptyset.
\end{equation*}
\end{corollary}
\begin{proof}
Take $c = 0$ in Theorem \ref{theo:reint_partial} and recall that $\lineality(( \stdCone ^1)^*) = (\stdCone^1)^\perp$.
\end{proof}

For more results taking into account partial polyhedrality see Chapter 20 of 
\cite{rockafellar} and Propositions 1 and 2 of \cite{lourenco_muramatsu_tsuchiya2}.

\section{Distance to polyhedrality, FRA-Poly and tightness}\label{sec:fra_poly}
Here we will discuss FRA-Poly, which is 
a facial reduction algorithm divided into two phases. The first detects infeasibility and 
restores strong duality, while the second finds the minimal face. For an example illustrating 
FRA-Poly, see Appendix \ref{app:ex}.

The idea behind the classical FRA is that whenever strong feasibility 
fails, we can obtain reducing directions until strong feasibility is satisfied 
again. Similarly, Phase 1 of FRA-Poly is based on
the fact that whenever the {\PPS} condition in Proposition \ref{prop:slater2}  fails, we may 
also obtain reducing directions until the {\PPS} condition is satisfied, thanks to Theorem \ref{theo:reint_partial}. 
After that, a single extra facial reduction step is enough to go to the minimal face.
 As the {\PPS} condition is weaker than full-on strong feasibility, FRA-Poly 
has better worst case bounds  in many cases.

We now present a disclaimer of sorts.
The theoretical results presented in this section and the next stand whether FRA-Poly is doable or 
not for a given $\stdCone$. 
If we wish to do facial reduction concretely (even if it is by hand!), we need to 
make a few assumptions on our computational capabilities 
and on our knowledge on the lattice of faces of $\stdCone$. 
First of all, we must be able to solve problems over faces of $\stdCone$ such that both the primal and the dual 
satisfy the {\PPS} condition and we must also be able to do 
basic linear algebraic operations.  Also, for 
each face $\stdFace$ of $\stdCone$ we must know:
\begin{enumerate}
	\item $\spanVec \stdFace$, 
	\item at least one point $\stdInt \in \reInt \stdFace$,
	\item at least one point $\stdInt^* \in \reInt \stdFace^*$,
	\item whether $\stdFace$ is polyhedral or not.
\end{enumerate}
We remark that 
apart from  knowledge about the polyhedral faces, our assumptions are not very different from what it is 
usually assumed \emph{implicitly} in the FRA literature. 
For symmetric cones, which 
include direct products of $\PSDcone{n}$, $\SOC{n}$ and $\Re^n_+$, they are reasonable since their 
lattice of faces is well-understood and every face is again a symmetric cone. 
So, for instance, $\stdInt$ can be taken as the identity element for 
the corresponding Jordan algebra. On the other hand, if $\stdCone$ is, say, the copositive 
cone $\mathcal{C}^n$, we might have some trouble fulfilling the requirements, inasmuch  as our knowledge of 
the faces of $\mathcal{C}^n$ is still lacking.

\subsection{Distance to Polyhedrality}
Here we introduce the notion of \emph{distance to polyhedrality}. In what follows, 
if we have a chain of faces $\stdFace _1 \subsetneq \ldots \subsetneq \stdFace _\ell$, the 
length of the chain is defined to be $\ell$.
\begin{definition}
	Let $\stdCone$ be a nonempty closed convex cone. The \emph{distance to polyhedrality} $\distP(\stdCone)$ is the 
	length \emph{minus one} of the longest strictly ascending chain of nonempty faces $\stdFace _1 \subsetneq \ldots \subsetneq \stdFace _{\ell} $ which satisfies:
	\begin{enumerate}
		\item  $\stdFace _1$ is polyhedral;  
		\item $\stdFace _j$ is not polyhedral for $j > 1$.
	\end{enumerate}
\end{definition}

The distance to polyhedrality is a well-defined concept, because the lineality space
of $\stdCone$ is always a polyhedral face of $\stdCone$.
Moreover, $\distP(\stdCone)$ counts the maximum number of facial reduction steps that 
can be taken before we reach a polyhedral face.
\begin{example}\label{ex:psd_lorentz}
	See section 2 and examples 2.5 and 2.6 in \cite{pataki_handbook} for more details 
	on the facial structure of $\PSDcone{n}$ and $\SOC{n}$.
	For the positive semidefinite cone $\PSDcone{n}$, we have $\distP(\PSDcone{n}) = n-1$.
	Liu and Pataki defined in \cite{LP17} \emph{smooth cones} as 
	full-dimensional, pointed cones (i.e., $\stdCone \cap -\stdCone = \{0\}$) such that any face that is not $\{0\}$ nor $\stdCone$ must be a half-line. For those cones we have $\distP(\stdCone) = 1$, when the dimension of 
	$\stdCone$ is	greater than $2$. Examples of smooth 
	cones include the Lorentz cone $\SOC{n}$ and the $p$-cones, for $1 < p < \infty$. For comparison, 
	the longest chain of nonempty faces of $\PSDcone{n}$ has length $n+1$ and 
	the one for any smooth 
	cone with dimension greater than 2 has length $3$. Note that  $\PSDcone{n}$ and $\SOC{n}$ are examples 
	of \emph{symmetric cones} and we mention in passing that a discussion about the longest chain of faces of a general symmetric cone can be found in Section 5.3 of \cite{IL16}.
	
\end{example}
\subsection{Strict complementarity in \eqref{eq:red} and \eqref{eq:red_dual}}\label{sec:compl}
The last ingredient we need is a discussion on the cases where 
jumping to $\minFaceD$ with a single facial reduction step is possible. 
Let $x^*$, $y^*$ be any pair of optimal solutions to \eqref{eq:primal} and \eqref{eq:dual}. Recall that if $\stdCone$ is $\Re^n_+,\PSDcone{n}$ or 
$\SOC{n}$, then $x^*,y^*$ are said to be \emph{strictly complementary} if 
the following equivalent conditions hold:
$$
s^* \in \reInt (\stdCone \cap \{x^*\}^\perp) \Leftrightarrow x^* \in \reInt (\stdCone^* \cap \{s^*\}^\perp) \Leftrightarrow \inProd{x^*}{s^*} = 0 \text{ and } x^* + s^* \in \reInt \stdCone,
$$
where $s^* = c - \stdMap ^\T y^*$.
For general $\stdCone$, these equivalencies may not hold and we might need 
to distinguish between primal and dual strict complementarity, see 
for instance, Definition 3.4 and Remark 3.6 in the chapter by Pataki \cite{pataki_handbook} and Equation (2.6) in Section 2 of the work by Tun\c{c}el and Wolkowicz \cite{TW12}. 
Based on those references, we will say 
that \eqref{eq:red} and \eqref{eq:red_dual}  satisfies  
\emph{(dual) strict complementarity} if $\opt{\text{\ref{eq:red}}} = \opt{\text{\ref{eq:red_dual}}}$
and there are optimal solutions $(x^*, t^*,w^*), (y_1^*,y_2^*,y_3^*)$ such that
 \begin{align}
 cy_1^* -\stdInt y_2^* -\stdMap^\T y_3^* &\in \reInt (\stdCone \cap\{ x^*\}^\perp)  \label{eq:strict1}\\
 t^* +  1-y_1^*(1+\inProd{c}{\stdInt^*}) + \inProd{\stdInt^* }{\stdMap ^\T y_3^*}&> 0  \label{eq:strict2}\\
 w^* + y_1^*-y_2^* & > 0. \label{eq:strict3} 
 \end{align}
 \vspace{-1\baselineskip}
\begin{proposition}\label{prop:strict}
Suppose $\opt{\text{\ref{eq:red}}} = \opt{\text{\ref{eq:red_dual}}} = 0$  and that  we have  optimal solutions to \eqref{eq:red} and 
\eqref{eq:red_dual} satisfying dual strict complementarity. 
If $w^* = 0$, then $\minFaceD = \stdCone \cap\{ x^*\}^\perp$.
\end{proposition}
\begin{proof}
Since $w^* = 0$, \eqref{eq:strict3} implies that $y_1^* > y_2^*$. Then, $\opt{\text{\ref{eq:red}}} = \opt{\text{\ref{eq:red_dual}}} $ and 
$\opt{\text{\ref{eq:red}}} \geq 0 $ implies that $y_2^* \geq 0$, so that $y_1^* > 0$ as well.
Therefore, from \eqref{eq:strict1} we obtain $c -\stdMap^\T \frac{y_3^*}{y_1^*} \in \reInt (\stdCone \cap\{ x^*\}^\perp).$
\end{proof}

Therefore, under strict complementarity, we can find $\minFaceD$ with a single facial reduction step.
Note that, here, we do not care about the choice of $\stdInt,\stdInt^*$.
For semidefinite programming, a similar observation was made in Theorem 12.28 of \cite{csw13}, where reducing directions are found through an 
auxiliary problem (AP). There, the authors show that a single direction is needed if and only if their AP satisfy strict complementarity. Another characterization of when one direction is enough can be found in Theorem 4.1 of \cite{DPW15}. One small advantage of \eqref{eq:red} and \eqref{eq:red_dual}  is that only linear constraints are used in addition to the conic constraints induced by $\stdCone$. In contrast, AP also adds quadratic constraints.

\subsection{FRA-Poly}

Henceforth, we will assume that $\stdCone$ is the product of $r$ cones and we will 
write $\stdCone = \stdCone ^1 \times \ldots \times \stdCone  ^r $. Then, recall that 
if $\stdFace$ is face of $\stdCone$, we can write $\stdFace = \stdFace ^1 \times \ldots \times \stdFace  ^r$, where $\stdFace ^i$ is a face of $\stdCone ^i$ for every $i$.

Consider the following 
FRA variant, which we call FRA-Poly.

\noindent
[{\bf Facial Reduction Poly - Phase 1}]

{\bf Input:} \eqref{eq:dual}

{\bf Output:} A set of reducing directions $\{d_1, \ldots, d_\ell\}$. If 
\eqref{eq:dual} is feasible, it outputs  some face $\stdFace \subseteq \stdCone$ 
for which the PPS condition holds, together 
with a dual slack $s'$ for which $s_j' \in \reInt \stdFace ^j $ for every $j$ such that $\stdFace ^j$  is nonpolyhedral.
If \eqref{eq:dual} is infeasible, the directions form a certificate of infeasibility.
\begin{enumerate}
\item $\stdFace _1 \leftarrow \stdCone, i \leftarrow 1 $
\item Let  $(x^*,t^*,w^*)$ and $(y_1^*,y_2^*,y_3^*)$ be any pair of optimal solutions to 
\eqref{eq:red} and \eqref{eq:red_dual} with
\begin{itemize}
	\item $\stdFace _i$ in place of $\stdCone $, where $\stdFace_i = \stdFace_i ^1 \times \ldots \times \stdFace_i  ^r$,
	\item any $\stdInt^* \in \reInt \stdFace _i^*$,
	\item any $\stdInt$  such that $\stdInt _j = 0$ if $\stdFace _i^j$ is polyhedral and 	$\stdInt _j \in \reInt \stdFace _i^j$, otherwise.
\end{itemize}

\item If $t^* = 0$ and  $\inProd{c}{x^*} < 0$, let $\minFaceD \leftarrow \emptyset$ and stop. \eqref{eq:dual}  is infeasible.
\item If $t^* = 0$ and $\inProd{c}{x^*} = 0$, let $d_i \leftarrow x^*, \stdFace _{i+1}  \leftarrow \stdFace _i  \cap \{d_i\}^\perp, i\leftarrow i+1$ and return to 2).
\item If $t^* > 0$, $s' \leftarrow  c  -\stdMap^\T  \frac{y_3^*}{y_1^*}$, $\stdFace \leftarrow \stdFace _i$ and stop. {\PPS} condition is satisfied.
\end{enumerate}

Note that Phase 1 of FRA-Poly might not end at the minimal face, but still, 
due to Proposition \ref{prop:slater2}, strong 
duality will be satisfied. First, we will prove the correctness of Phase 1, 
which essentially follows from Lemma \ref{lemma:red2}.
\begin{proposition}\label{prop:phase1}
The following hold.
\begin{enumerate}[label=(\roman*)]
	\item if \eqref{eq:dual} is feasible, then the output face $\stdFace$ satisfies 
	$\minFaceD \subseteq \stdFace$. Moreover, $s'$ is a dual 
feasible slack such that $s_j' \in \reInt \stdFace ^j $  for every $j$ such that $\stdFace ^j$  is nonpolyhedral, i.e., the PPS condition is satisfied for $\stdFace$.

In this case, Phase 1  stops after finding at most $\sum _{i=1}^{r} \distP(\stdCone ^i) $ directions. 
	\item \eqref{eq:dual} is infeasible if and only if Step 3 is reached. In this case, Phase 1 stops after finding 
	at most $1+ \sum _{i=1}^{r} \distP(\stdCone ^i) $ directions.
\end{enumerate}
\end{proposition}
\begin{proof}
We will focus on the statements about the bounds, since the other statements 
are direct consequences of 
Lemma \ref{lemma:red2}. Note that whenever Step 4 is reached, we have 
$\stdFace _{i+1} \subsetneq \stdFace _i$, since 
$x^*_j \not \in (\stdFace _i^j)^\perp$ for at least one nonpolyhedral cone $\stdFace _i^j$, due to item $(ii)$-$(b)$ of Lemma \ref{lemma:red2}. 
Therefore, whenever a new (proper) face is found, it is because 
we are making progress towards a polyhedral face for at least one nonpolyhedral cone. 	
	
By definition, after finding $\hat \ell = \sum _{i=1}^{r} (\distP(\stdCone ^i)) $ directions, 
$\stdFace _{\hat \ell + 1}$ is polyhedral. We now consider what happens if 
the algorithm has not stopped after all these directions were found.
In this case, when it is time to build \eqref{eq:red} and \eqref{eq:red_dual} with $\stdFace_{l+1}$ in
place of $\stdCone$ at Step 2, Phase 1 selects $e = 0$ and $\stdInt^* \in (\stdFace _{\hat \ell + 1})^*$.

First, suppose that \eqref{eq:dual} is feasible and  let 
$y$ be such that $c-\stdMap^\T y \in \feasS$. 
Since $\stdInt = 0$, $(\alpha, \alpha, \alpha y)$
is a feasible solution for \eqref{eq:red_dual}, when $\alpha>0$ is sufficiently small. It follows 
that $\opt{\text{\ref{eq:red_dual}}} > 0$ and that Phase 1 eventually reaches Step 5. This gives item $(i)$.

Finally, suppose that \eqref{eq:dual} is infeasible. By Lemma \ref{lemma:red2},  $\opt{\text{\ref{eq:red_dual}}} = 0$. Since $\stdInt = 0$,  \eqref{eq:red:2} implies that every optimal solution of 
\eqref{eq:red} will be a triple of the form
$(x^*,0,1)$, which implies that Step 3 will be reached and a single new direction will be 
added. This gives  item $(ii)$.
\end{proof}

\begin{remark}
Let $\ell$ be the number of directions found in 
Phase 1. 
When \eqref{eq:dual} is feasible, $\ell + 1$ is the total number of times the pair of problems \eqref{eq:red}, \eqref{eq:red_dual} are solved during Phase 1. After solving \eqref{eq:red}, \eqref{eq:red_dual} $\ell$ times, a face $\stdFace$ of $\stdCone$ satisfying the {\PPS} condition is computed. 
However, it is necessary to solve \eqref{eq:red}, \eqref{eq:red_dual} one extra time to reach the stopping criteria in Step 5 and obtain $s'$, which is a certificate that $\stdFace$ indeed satisfies the {\PPS} condition. 
When \eqref{eq:dual} is infeasible, $\ell$ coincides with the 
number of times problems \eqref{eq:red}, \eqref{eq:red_dual} are solved during Phase 1.
\end{remark}
If we merely want a face of $\stdCone$ satisfying the {\PPS} condition, we can stop at Phase 1. 
In any case, we will now show that is possible to jump directly to the minimal face in a single facial reduction step. 

[{\bf Facial Reduction Poly - Phase 2}]

{\bf Input:} \eqref{eq:dual}, the output of Phase 1: $\stdFace$ and $s'$, with $\stdFace \neq \emptyset$.

{\bf Output:} $\minFaceD$, a dual feasible slack $\hat s \in \reInt \minFaceD$ and, perhaps, 
an extra reducing direction $d$.
\begin{enumerate}
\item Let $\hat \stdCone = {\hat \stdCone} ^1\times \ldots \times {\hat \stdCone} ^r$ such that ${\hat \stdCone} ^j = \stdFace ^j$ if 
$\stdFace ^j$ is polyhedral and $\hat \stdCone ^j = \spanVec \stdFace^j$ otherwise. 
\item Let $(x^*,t^*,w^*)$ and $(y_1^*,y_2^*,y_3^*)$ be
any pair of \emph{strictly complementary} optimal solutions to \eqref{eq:red} and \eqref{eq:red_dual} with
\begin{itemize}
	\item $\hat \stdCone$ in place of $\stdCone$,
	\item any $\stdInt \in \reInt \hat \stdCone$,
	\item any $\stdInt^* \in \reInt {\hat \stdCone}^* $.
\end{itemize}
\item If $t^* = 0$, let $d \leftarrow x^*$, $\minFaceD \leftarrow \stdFace \cap \{x^*\}^\perp$.
Let $\tilde s$ be $c  -\stdMap^\T  \frac{y_3^*}{y_1^*} $. Then, we let 
$\hat s$ be a convex combination of $\tilde s$ and $s'$ such that $\hat s \in \reInt \minFaceD$ and stop.\footnote{\label{foot}In Proposition \ref{prop:min_face} it is shown that if $z_\beta = \beta s' + (1-\beta)\tilde s $ is such that $\beta \in (0,1)$ and $\beta$ is sufficiently close to $1$ then $z_\beta \in  \reInt \minFaceD$. Therefore, after determining $\minFaceD$, $\hat s$ can be found through a simple search procedure, e.g., the bisection method.} 
\item If $t^* > 0$,  $\minFaceD \leftarrow \stdFace$.
Let $\tilde s$ be $c  -\stdMap^\T  \frac{y_3^*}{y_1^*} $. Then, we let 
$\hat s$ be a convex combination of $\tilde s$ and $s'$ such that $\hat s \in \reInt \minFaceD$ and stop.\footref{foot}

\end{enumerate}

Note that in Phase 2, the cone $\hat{\stdCone}$ is polyhedral, therefore, 
both \eqref{eq:red_dual} and \eqref{eq:red} are 
polyhedral problems. Therefore,  strictly complementary solutions are ensured to exist, which is a 
consequence of Goldman-Tucker Theorem and also follows from the results of 
McLinden \cite{McLinden82} and Akg\"ul \cite{Akgul84}. 
We also remark that a strictly complementary solution of a polyhedral problem can be found by solving 
a single linear program, see, for instance, the article by Freund, Roundy and Todd \cite{FRT85} and the related 
work by Mehrotra and Ye \cite{MY93}.

We now try to motivate the next proposition. At Phase 2, a single facial reduction iteration is performed. In the usual facial reduction approach, we would build the problems \eqref{eq:red_dual} and \eqref{eq:red}
using $\stdCone = \stdFace$ and seek a reducing direction belonging to $\stdFace^*$. The subtle point in Phase 2 is that we
use $\hat \stdCone$ in place of $\stdFace$, which is potentially larger since the nonpolyhedral blocks were relaxed to their span. This restricts our search for reducing directions to 
$\hat \stdCone^*$, which is potentially smaller than $\stdFace^*$. However, the proof in Proposition \ref{prop:min_face} will show that, at this stage, 
any reducing direction must be already confined to $\hat \stdCone^*$.

\begin{proposition}\label{prop:min_face}
The output face of Phase 2 is $\minFaceD$ and there exists $\hat s$ as in Steps 3. and 4.
\end{proposition}
\begin{proof}
Suppose that the output face $\stdFace$ of Phase 1 satisfies $\stdFace \neq \minFaceD$. 
By Lemma 3.2 in \cite{article_waki_muramatsu}, there is 
a reducing direction $x$ such that $x \in \stdFace ^* \cap \ker \stdMap \cap \{c\}^\perp$ 
and $x \not \in \stdFace ^\perp$.
Due to Proposition \ref{prop:phase1}, any such reducing direction $x$ satisfies $\inProd{x}{s'} = 0$, which implies that 
 $x_j \in ({\stdFace ^j})^\perp = ({\spanVec \stdFace ^j})^\perp$ for every $j$ such that $\stdFace ^j$ is not polyhedral, 
 since $s_j' \in \reInt \stdFace ^j $ for those $j$.
Therefore, the possible reducing directions are confined to the polyhedral cone $\hat \stdCone^*$,
where $\hat \stdCone$ is the cone in Step 1. of Phase 2.

Since $\hat \stdCone$ is polyhedral, the problems \eqref{eq:red} and \eqref{eq:red_dual} 
 are polyhedral and they admit strictly complementary optimal 
solutions $(x^*,t^*,w^*)$, $(y_1^*,y_2^*,y_3^*)$.
The fact that $x \not \in \stdFace^\perp$ implies that $\inProd{\stdInt}{x}\neq 0$ so that $(x/\inProd{\stdInt}{x},0,0) $ is an optimal 
solution to \eqref{eq:red}. Therefore, $t^* = y_2^* = 0$. Moreover, since \eqref{eq:dual} is 
feasible, we have $w^* = 0$.
By Proposition \ref{prop:strict}, we have 
$c -\stdMap^\T  \frac{y_3^*}{y_1^*} \in \reInt (\hat \stdCone \cap \{x^*\}^\perp).$

Let $\tilde s = c -\stdMap^\T  \frac{y_3^*}{y_1^*}$. Note that $\stdFace \cap \{x^*\}^\perp$ is a face of 
$\stdFace$ containing $\feasS$, since we argued that $x^*$ must be a reducing direction. We will prove that $\minFaceD = \stdFace \cap \{x^*\}^\perp$ by 
 showing that some convex combination of $s'$ and $\tilde s$ lies in $\reInt (\stdFace \cap \{x^*\}^\perp)$. 

Let $z_\beta = \beta s' + (1-\beta)\tilde s$. For 
all $\beta \in (0,1)$ and all $j$ such that $\stdFace ^j$ is polyhedral, we 
have $(z_\beta)_j \in \reInt ( \stdFace ^j  \cap \{x_j^*\}^\perp)$, because $\tilde s _j \in  \reInt (\stdFace ^j  \cap \{x_j\}^\perp)$ and 
$s'$ is feasible.
If $\stdFace ^j$ is not polyhedral, then $\stdFace ^j  \cap \{x_j^*\}^\perp = \stdFace ^j$, 
since $x_j \in (\stdFace ^j)^\perp$. Because $\tilde s _j \in \spanVec{\stdFace ^j}$ and 
$s_j' \in \reInt \stdFace^j$, for all $\beta $ sufficiently close to $1$ we have 
$(z_\beta)_j \in \reInt \stdFace^j$. Therefore, it is possible to select $\beta \in (0,1)$ 
such that $(z_\beta)_j \in \reInt (\stdFace ^j  \cap \{x_j\}^\perp) $ for all $j$. This 
shows that $\minFaceD = \stdFace \cap \{x^*\}^\perp $.

If $\stdFace$ was already the minimal face to begin with, then 
$t^* > 0$. We can then  proceed in a 
similar fashion. The only difference is that due to \eqref{eq:red_dual:1}, we 
will have that $\tilde s = c -\stdMap^\T  \frac{y_3^*}{y_1^*} $ satisfies 
$\tilde s _j \in  \reInt (\stdFace ^j)$ for every $j$ such that $\stdFace ^j$ is polyhedral.
And as before, we can select a convex combination of $s'$ and $\tilde s$ belonging to 
the relative interior of $\minFaceD$.
\end{proof}

We then arrive at the main result of this section.
\begin{theorem}\label{theo:fra_bound}
Let $\stdCone = \stdCone ^1\times \ldots \times \stdCone^r$. The minimum face $\minFaceD$ 
that contains the feasible region of \eqref{eq:dual} can be found in no more than 
$1+ \sum _{i=1}^{r} \distP(\stdCone ^i) $ facial reduction steps.
\end{theorem}
\begin{proof}
If \eqref{eq:dual} is infeasible, then $\minFaceD = \emptyset$ and the result 
follows from Proposition \ref{prop:phase1}. So suppose now that 
\eqref{eq:dual} is feasible. Then Phase 1 ends after finding at most 
$\sum _{i=1}^{r} \distP(\stdCone ^i)$ directions. Due to Proposition \ref{prop:min_face}, at most one 
extra direction is needed to jump to the minimal face.
\end{proof}


Recall that $\ell _\stdCone$ is the length of the longest chain of 
strictly ascending nonempty faces of $\stdCone$.
If one uses the ``classical'' facial reduction approach, it takes no more than $\ell _\stdCone - 1$
facial reduction steps to find the minimal face, when \eqref{eq:dual} is feasible. 
See, for instance, Theorem 1 in \cite{pataki_strong_2013} or  Corollary 3.1 in \cite{article_waki_muramatsu}.
When $\stdCone$ is a direct product of several 
cones, we have $\ell _\stdCone  = 1 + \sum _{i=1}^{r} (\ell _{\stdCone^i} -1)$. We will 
end this subsection by showing that, under the relatively weak hypothesis 
that $\stdCone^i$ is not a subspace, we have $\distP(\stdCone ^i) < \ell _{\stdCone^i } - 1$. This means that FRA-Poly compares favorably to the classical FRA analysis  and the 
difference between the two bounds grows at least linearly with the number of cones.

\begin{theorem}\label{theo:dist_comp2}
If $\stdCone$ is not a subspace then $1 + \distP(\stdCone) \leq \ell _{\stdCone }-1$.
In particular, if $\stdCone $ is the direct product of $r$ closed convex 
cones that are not subspaces we have:
\begin{equation*}
 1 + r + \sum _{i=1}^{r} \distP(\stdCone ^i)   \leq 1 + \sum _{i=1}^{r} (\ell _{\stdCone^i} -1).
\end{equation*}
\end{theorem}
\begin{proof}
Let $U = \lineality \stdCone$. Then we have $\stdCone = (\stdCone\cap U^\perp) + U$.
If we let $\hat \stdCone =  \stdCone\cap (U^\perp) $, we have that 
$\lineality (\hat \stdCone ) = \{0\}$ so that $\hat \stdCone$ is pointed and 
$\hat \stdCone \neq \{0\}$ if $\stdCone$ is not a subspace. Recall that the minimal nonzero face of any nonzero 
pointed cone must be an extreme ray, i.e., an one dimensional face.
Therefore,  the first two faces of any longest chain of faces of $\hat \stdCone$ must be $\{0\}$ 
and some extreme ray. Therefore, we have 
$1 + \distP(\hat \stdCone) \leq \ell _{\hat \stdCone} -1 $. 

Note that there is a bijection 
between the faces of $\stdCone$ and the set $\{\stdFace + U \mid \stdFace \text{ is a face of } \hat \stdCone \}$. A similar 
 correspondence holds between the polyhedral faces of $\stdCone$  and the set $\{\stdFace + U \mid \stdFace \text{ is a polyhedral face of } \hat \stdCone \}$.
 Therefore, $\ell _\stdCone = \ell _{\hat \stdCone}$ and $ \distP(\stdCone) = \distP(\hat \stdCone)$.
 This shows that $1 + \distP(\stdCone)  \leq \ell _{\stdCone }-1$. 
 To conclude, note that if $\stdCone$ is a direct product of $r$ cones then
 $\ell _{\stdCone} = 1 + \sum _{i=1}^{r} (\ell _{\stdCone^i} -1)$, so the result 
 follows from applying what we have done so far to each $\stdCone^i$.
\end{proof}

\subsection{Tightness of the bound}\label{sec:tight}
It is 
reasonable to consider whether there are instances that actually need 
the amount of steps predicted by Theorem \ref{theo:fra_bound}. In this section we will 
take a look at this issue. The following notion will be helpful.

\begin{definition}[Singularity degree]\label{def:sing}
	Consider the set  of possible outputs  $\{d_1,\ldots$ $, d_\ell\}$ of 
	the Generic Facial Reduction algorithm in Section \ref{sec:fra}.
	The singularity degree of \eqref{eq:dual} is the minimum $\ell$ among 
	all the possible outputs and is denoted by $d(D)$.
\end{definition}

That is, the singularity degree is the minimum number of facial reduction steps before $\minFaceD$ is found. In the recent work by Liu and Pataki \cite{LP17}, there is also an equivalent definition of singularity degree for 
feasible problems, see Definition 6 therein.
As far as we know, the expression ``singularity degree'' in this context is due to 
Sturm  in \cite{sturm_error_2000}, where he showed the connection between the singularity 
degree of a positive semidefinite program and  error bounds, see also \cite{sturm_handbook}.

The singularity degree of \eqref{eq:dual} is a quantity that depends on 
$c, \stdMap$ and $\stdCone$. The classical facial reduction 
strategy gives the bounds $d(D) \leq \ell_\stdCone -1$ when \eqref{eq:dual} is 
feasible and $d(D) \leq \ell_\stdCone$ when \eqref{eq:dual} is infeasible.
Theorem \ref{theo:fra_bound} readily implies that 
$d(D) \leq 1+ \sum _{i=1}^{r} \distP(\stdCone ^i)$, no matter whether \eqref{eq:dual} is 
feasible or not. Due to Theorem \ref{theo:dist_comp2}, this bound is likely 
to compare favorably to  $\ell_\stdCone -1 =  \sum _{i=1}^{r} (\ell _{\stdCone^i} -1)$.

Tun\c{c}el constructed an SDP instance with singularity degree $d(D) = n-1 = \ell _{\PSDcone{n}}-2$, 
see Section 2.6 in \cite{tuncel_polyhedral_2010} or the section ``Worst case instance'' in \cite{csw13}. Now, 
let 
$\stdCone = \SOC{t_1} \times \ldots\times \SOC{t_{r_1}} \times \PSDcone{n_1}  \times \ldots \times \PSDcone{n_{r_2}}$ be
the direct product of $r_1$ second order (Lorentz) cones and $r_2$ positive semidefinite cones.
In this case, Theorem \ref{theo:fra_bound} implies that
$
d(D) \leq 1 + r_1 + \sum _{j = 1}^{r_2} (n_j  - 1).
$ In Appendix \ref{app:ex} we  extend Tun\c{c}el's example and show that for every such $\stdCone$ there is a feasible instance for which 
$ d(D) = r_1 + \sum _{j = 1}^{r_2} (n_j  - 1)$, thus showing the worst case bound in Theorem \ref{theo:fra_bound} is off by at most one. 

This type of $\stdCone$ was also studied by Luo and Sturm in \cite{sturm_handbook}, 
where they discussed a regularization procedure which ends in at most $r_1 + \sum _{j = 1}^{r_2} (n_j  - 1)$ steps, see Theorem 7.4.1 therein. However, their definition of regularity does not imply strong feasibility, so similarly to Phase 1 of FRA-Poly, an additional step is necessary  (akin to a facial reduction step)  before the minimal face is reached, see Lemma 7.3.3. In total we get the same 
bound predicted by Theorem \ref{theo:fra_bound}.

We remark that we were unable to  construct a feasible instance with singularity degree $1+r_1 + \sum _{j = 1}^{r_2} (n_j  - 1)$. Note that if $\stdCone = \PSDcone{n}$, 
since each facial reduction step reduces the possible ranks of feasible matrices, if we need 
$n$ steps it is because $\feasS = \{0\}$. But if $\feasS = \{0\}$, then $c \in \matRange \stdMap^\T$ and Gordan-Stiemke's Theorem implies
the existence of $d \in (\reInt \stdCone^*) \cap \ker \stdMap$. Therefore, we can go to $\minFaceD$ with a single step, since 
$\stdCone \cap \{d\}^\perp = \{0\}$. So, in fact, 
we never need more than $n-1$ steps for feasible SDPs and Tun\c{c}el's example is indeed the worse that could happen in this case. 
A similar argument holds when 
$\stdCone = \SOC{n}$, where we never need more than a single step if we select the directions optimally.
But when we have direct products,  the possible interactions between the blocks makes it hard to argue 
that the $+1$ is unnecessary, although the partial polyhedral Gordan-Stiemke theorem (Corollary \ref{col:partial_gordan})  helps rule out a few 
cases.

 We will now take a  look at what could be done  to ensure that a facial reduction 
algorithm never takes more steps than the necessary to find $\minFaceD$. Consider the Generic Facial Reduction algorithm in Section \ref{sec:fra}. All the 
directions, with the possible exception of the last, belong to $\stdFace_i^*\cap \ker \stdMap \cap \{c\}^\perp$.
In particular, the FRAs considered in \cite{sturm_error_2000,sturm_handbook} and the FRA-CE variant in \cite{article_waki_muramatsu} always select the most interior direction 
possible. In our context, this means that whenever step $2.i)$ is reached the following choice is made:
\begin{equation}
d_i \in \reInt (\stdFace _i ^*\cap \ker \stdMap \cap \{c\}^\perp). \label{eq:dir_ri}
\end{equation}
In fact,  the singularity degree was originally  defined not as in Definition \ref{def:sing},
but as the number of steps that 
their particular algorithms take to find the minimal face\footnote{There is a minor incompatibility between the two definitions. Sturm considered that a problem with $\feasS = \{0\}$ has singularity degree $0$, see Step 1 in Procedure 1 in \cite{sturm_error_2000}. 
According to the 
definition in \cite{LP17} and our own, such a problem would have singularity degree $1$.}. Although intuitive, 
it is not entirely obvious that the choice in \eqref{eq:dir_ri} minimizes the number of directions,
so let us take a look at this issue.

\begin{proposition}
Suppose that \eqref{eq:dual} is feasible and that at each step of the Generic Facial 
Reduction algorithm $d_i$ is selected as in \eqref{eq:dir_ri}. Then, the algorithm 
will output exactly $d(D)$ directions.
\end{proposition}
\begin{proof}
Suppose $d(D) > 0$ and let $(d_1,\ldots, d_\ell)$ be a sequence of reducing directions such that 
$\ell = d(D)$ and the last face is $\minFaceD$. Let $d_1^* \in \reInt (\stdCone ^*\cap \ker \stdMap \cap \{c\}^\perp)$. 

 Since $d_1^*$ is a relative interior point, there is  $\alpha > 1$ such 
 that $\alpha d_1^* + (1-\alpha)d_1 \in \stdCone^*\cap \ker \stdMap \cap \{c\}^\perp$, see Theorem 6.4 in \cite{rockafellar}.
 Now, let  $x \in \stdCone \cap \{d_1^*\}^\perp $. We must have 
 $$
 \inProd{x}{\alpha d_1^* + (1-\alpha)d_1 } \geq 0. 
 $$
 Since $\inProd{x}{d_1^*} = 0$ and $(1-\alpha) < 0$, we have $\inProd{x}{d_1} = 0$ as well. That is, we have
\begin{equation}
\stdCone \cap \{d_1^*\}^\perp \subseteq \stdCone \cap \{d_1\}^\perp.\label{eq:prop:ri_dir}
\end{equation}
(Note that this shows that if $d_1 \not \in \stdCone ^\perp$ then $d_1^* \not \in \stdCone^\perp$ as well.)
Since taking the dual cone inverts the containment, we have 
$$
(\stdCone \cap \{d_1^*\}^\perp \cap \cdots \cap \{d_i\}^\perp)^* \supseteq (\stdCone \cap \{d_1\}^\perp \cap \cdots \cap \{d_i\}^\perp)^*,
$$
for every $i$. Therefore, $(d_1^*,\ldots, d_\ell)$ is still a valid sequence of reducing directions for 
\eqref{eq:dual} and the corresponding chain of faces  still ends in the minimal face, due to \eqref{eq:prop:ri_dir}. 
Likewise, we substitute $d_2$ by $d_2 ^*$ following \eqref{eq:dir_ri} with 
$\stdFace _2 = \stdCone \cap \{d_1^*\}^\perp$ and proceed 
inductively. This shows that selecting according to \eqref{eq:dir_ri} does indeed 
produce the least number of directions.
\end{proof}

We remark that the argument that leads to \eqref{eq:prop:ri_dir} also shows that if $d_1$ was already chosen according to 
\eqref{eq:dir_ri}, we would have in fact $\stdCone \cap \{d_1^*\}^\perp = \stdCone \cap \{d_1\}^\perp$. So that if we use 
the choice in \eqref{eq:dir_ri} the resulting chain of faces is unique even if the directions themselves are not.

For some cases, we can expect to implement the choice in \eqref{eq:dir_ri}. 
If  \eqref{eq:dual} and \eqref{eq:primal} are both 
strongly feasible and $\stdCone = \PSDcone{n}$, then it is known that the central path converges 
to a solution that is a relative interior point of the set of optimal solutions \cite{KRT97} and 
the facial reduction approach in \cite{PFA15} uses this fact in an essential way.
Therefore, the choice in \eqref{eq:dir_ri} might be implementable in the context of 
interior point methods although it is not known whether for other algorithms, say augmented Lagrangian methods, a similar property holds.
Still, as interior point methods are very revelant to conic linear programming, one of the referees prompted us to prove the following. 

\begin{proposition}
Let $\stdInt \in \stdCone, \stdInt^* \in \stdCone$ and 
let $\Omega$ denote the optimal solution set of \eqref{eq:red}.
Let  $(x^*,t^*,w^*) \in \reInt \Omega$. If $t^* = w^* = 0$, then  
$$
x^* \in \reInt (\stdCone ^*\cap \ker \stdMap \cap \{c\}^\perp).
$$
\end{proposition}
\begin{proof}
Let $P_x$ be the linear map that takes $(x,t,w) \in \Re^n \times \Re \times \Re$ to $x$.
Since at optimality we have $t^* = w^* = 0^*$, Equation \eqref{eq:red:2} implies that we have
$$
\Omega _x = \stdCone ^*\cap \ker \stdMap \cap \{c\}^\perp \cap H,
$$
where $\Omega _x = P_x(\Omega) $ and $H = \{x \in \Re^n \mid \inProd{e}{x} = 1\}$.
As $P_x$ is linear, we have $P_x(\reInt \Omega) = \reInt \Omega _x$, see 
Theorem 6.6 in \cite{rockafellar}. Therefore,  $(x^*,t^*,w^*) \in \reInt \Omega$ implies $x^* \in \reInt \Omega _x$.

Let $C = \stdCone ^*\cap \ker \stdMap \cap \{c\}^\perp $, so that $\Omega _x = C\cap H$.
Note that the proposition will be proved if we show that $\reInt \Omega _x = (\reInt (C)) \cap H$.

First, observe that since $H$ is an affine space, we have $\reInt H = H$. Then, by 
Theorem 6.5 in \cite{rockafellar}, a sufficient condition for $(\reInt (C)) \cap H = 
\reInt(C\cap H)$ to hold is that $(\reInt (C) )\cap H \neq \emptyset $. We will now construct a point in $(\reInt (C) )\cap H$.

Let $z \in \reInt(C)$. Note that $x^* \in C$ as well, so there is 
$\alpha > 1$ such that $\alpha z + (1-\alpha)x^* \in C$, by Theorem 6.4 in \cite{rockafellar}. Then, $\inProd{\stdInt}{\alpha z + (1-\alpha)x^*} \geq 0 $ together with $(1-\alpha ) < 0$ and $\inProd{e}{x^*} = 1$ implies that $\inProd{e}{z} > 0$. 
Therefore, $\frac{z}{\inProd{e}{z}} \in (\reInt (C)) \cap H$.
This shows that
$
\reInt \Omega _x = (\reInt (C)) \cap H.
$
\end{proof}

\section{Applications of FRA-Poly}\label{sec:app}
In this section, we discuss applications of FRA-Poly. In the first one, we sharpen 
a result proven by Liu and Pataki \cite{LP17} on the geometry of weakly infeasible problems. 
In the second, we show that the singularity degree of problems over the doubly nonnegative cone 
is at most $n$.

As mentioned before, the singularity degree only depends on $c,\stdMap$ and $\stdCone$. 
Finding the minimal face $\minFaceD$ ensures that no matter which $b$ we select, as long as 
$\dOpt$ is finite, there will be zero duality gap and primal attainment. This suggests 
the following definition that also depends on $b$ and, thus, produce a less conservative 
quantity. 

\begin{definition}[Distance to strong duality]\label{def:dist_str}
The distance to strong duality $\distS(D)$ is the minimum number of facial reduction 
steps (at \eqref{eq:dual}) needed to ensure $\opt{\hat P} = \dOpt$, where $(\hat P)$ is the problem 
$\inf \{\inProd{c}{x} \mid \stdMap x = b, x \in \stdFace _{\ell+1} ^* \}$ and 
$\stdFace _{\ell+1}$ is a face obtained after a sequence of $\ell$ facial reduction steps.
If  $-\infty < \dOpt < +\infty$, we also require attainment of $\opt{\hat P}$.

Similarly, we define $\distS(P)$ as the minimum number of facial reduction steps needed 
to ensure that  $\pOpt = \opt{\hat D}$ and that $\opt{\hat D}$ is attained when $-\infty < \pOpt < +\infty$, where $(\hat D)$ is the  problem in dual standard form arising after some sequence of facial reduction 
steps is done at \eqref{eq:primal}.
\end{definition}

Clearly, we have $\distS(D) \leq d(D)$. However, since Phase 1 of FRA-Poly restores 
strong duality in the sense of Definition \ref{def:dist_str}, we obtain the 
nontrivial bound $\distS(D) \leq \sum_{i=1}^r \distP(\stdCone^i)$.

\subsection{Weak infeasibility}\label{sec:wi}
Let $\stdAffine $ denote the affine space $ c + \matRange{\stdMap^\T }$ and let 
the tuple $(\stdAffine, \stdCone)$ denote the  feasibility problem of seeking 
an element in the intersection $\stdAffine \cap \stdCone = \feasS$. 
In \cite{lourenco_muramatsu_tsuchiya},
we showed that if $\stdCone = \PSDcone{n}$ and \eqref{eq:dual} is weakly infeasible, then there is an affine subspace $\stdAffine '$ contained 
in $\stdAffine$ of dimension at most $n -1$ such that $(\stdAffine ', \stdCone)$ is also weakly infeasible. 
This can be interpreted as saying that ``we need at most $n-1$ directions to approach the positive semidefinite cone''. 
In \cite{LP17}, Liu and Pataki generalized this result and proved that those affine spaces  always 
exist and $\ell _{\stdCone^*} - 1$ is  an upper bound for the dimension of $\stdAffine '$, 
see Corollary 1 therein.
We proved a bound of $r$  for the direct product of $r$ Lorentz 
cones \cite{lourenco_muramatsu_tsuchiya2},  which is tighter than the one in \cite{LP17}.
Here we will refine these results. Consider the following pair of problems.
\vspace{-3\baselineskip}
\begin{multicols}{2}
	\begin{align}
	\underset{x}{\inf} & \quad \inProd{c}{x} \label{eq:primal_inf}\tag{$\text{P}_{\text{feas}}$}\\ 
	\mbox{subject to} & \quad \stdMap x = 0 \nonumber \\ 
	&\quad x \in \stdCone ^* \nonumber 
	\end{align}
	\break
	\begin{align}
	\underset{y}{\sup} & \quad 0 \label{eq:dual_inf} \tag{$\text{D}_{\text{feas}}$} \\ 
	\mbox{subject to} & \quad c - \stdMap ^\T y \in \stdCone. \nonumber\\ \nonumber
	\end{align}
\end{multicols}
\vspace{-1\baselineskip}
Recall that strong infeasibility of \eqref{eq:dual} is equivalent to the existence of $x$ such that 
$x \in \stdCone^*\cap \ker \stdMap$ and $\inProd{c}{x} < 0$, see Lemma 5 in \cite{Luo97dualityresults}.
Therefore,  $\opt{\text{\ref{eq:primal_inf}}} = -\infty$ if 
and only if \eqref{eq:dual} is strongly infeasible. It follows that  \eqref{eq:dual} is weakly infeasible if and only if $\opt{\text{\ref{eq:primal_inf}}} $ is zero and $\opt{\text{\ref{eq:dual_inf}}} = -\infty$.

When $\opt{\text{\ref{eq:primal_inf}}} = 0$ and we restore strong duality to \eqref{eq:primal_inf} in the sense of Definition \ref{def:dist_str}, 
a feasible solution will appear at the dual side. Even if that solution is not feasible for the original problem \eqref{eq:dual},
it will give us some information about \eqref{eq:dual} and this is the motivation  behind Theorem \ref{theo:wi} below.

We first need an auxiliary result that shows that if \eqref{eq:dual} is strongly infeasible and 
we try to regularize \eqref{eq:primal_inf}, then \eqref{eq:dual_inf} (and, therefore, \eqref{eq:dual}) will stay strongly infeasible.

\begin{lemma}\label{lema:str_inf_preserv}
Let $d$ be a reducing direction for \eqref{eq:primal_inf}, i.e., 
$d \in (\matRange \stdMap^\T)  \cap \stdCone$.  Let
 $\hat \stdCone = (\stdCone^*\cap \{d\}^\perp)^*$. Then \eqref{eq:dual_inf} is strongly 
infeasible if and only if $(\hat{\text{D}}_{\text{feas}})$ is strongly infeasible, where $(\hat{\text{D}}_{\text{feas}})$ is the problem 
with $\hat \stdCone$ in place of $\stdCone$.
\end{lemma}
\begin{proof}
$(\Leftarrow)$ This part is clear, since $\stdCone \subseteq \hat \stdCone$.
	
	$(\Rightarrow)$ Strong infeasibility of \eqref{eq:dual_inf}
	is equivalent to the existence of 
	$x$ such that $x \in \stdCone^*\cap \ker \stdMap$ and $\inProd{c}{x} < 0$. 
	Since 
	$d \in  \matRange \stdMap^\T $, we have $\inProd{x}{d} = 0$. By the same principle, 
	$x$ induces strong infeasibility for $(\hat{\text{D}}_{\text{feas}})$ as well. 
\end{proof}
 
\begin{theorem}\label{theo:wi}
\hfill 
\begin{enumerate}[label=(\roman*)]
	\item  \eqref{eq:dual} is not strongly infeasible if and only if there are: 
\begin{enumerate}[label=({\alph*})]
\item a sequence of 
reducing directions $\{d_1, \ldots , d_\ell\}$ for \eqref{eq:primal_inf} restoring strong duality 
in the sense of Definition \ref{def:dist_str} with $\ell = \distS(\text{\ref{eq:primal_inf}})$ and
\item $\hat y$ such that $c - \stdMap^\T \hat y \in (\stdCone^*\cap \{d_1\}^\perp \cap \cdots \cap \{d_\ell\}^\perp)^*$.
\end{enumerate}
\item If \eqref{eq:dual} is not strongly infeasible, there is 
an affine subspace $\stdAffine ' \subseteq c - \matRange \stdMap^\T$ such that 
$(\stdAffine ', \stdCone)$ is not strongly infeasible and the dimension of $\stdAffine '$ satisfies 
\begin{equation*}
\dimSpace(\stdAffine ') \leq \distS(P') \leq \sum _{i=1}^{r} \distP((\stdCone ^i)^*).
\end{equation*}
In particular, if \eqref{eq:dual} is weakly infeasible, then $(\stdAffine ', \stdCone)$ is 
weakly infeasible.
\end{enumerate}
\end{theorem}

\begin{proof}
\begin{enumerate}[label=({\it\roman*}),wide, labelwidth=!]
	\item $(\Rightarrow)$ Due to the assumption that  \eqref{eq:dual} is not strongly infeasible, we have $\opt{\text{\ref{eq:primal_inf}}} = 0$.
Now, let $\{d_1, \ldots , d_\ell \}$ be a sequence of reducing directions for \eqref{eq:primal_inf} that 
restores strong duality in the sense of Definition \ref{def:dist_str} with 
$\ell = \distS(\text{\ref{eq:primal_inf}})$. 
Let $\hat \stdFace _{\ell+1} = \stdCone^*\cap \{d_1\}^\perp \cap \cdots \cap \{d_\ell\}^\perp$.

Now, \eqref{eq:primal_inf} shares the same feasible region with the problem 
\begin{equation*}
 \inf\, \{ \inProd{c}{x} \mid \stdMap x = 0, x \in \hat \stdFace _{\ell+1} \}. \tag{$\text{\^ P}$}
\end{equation*}
Since facial 
reduction  preserves the optimal value, we have $\opt{\hat P} = \opt{\text{\ref{eq:primal_inf}}} = 0 $. 
Because the  reducing directions restore strong duality, we have $\opt{\hat D} = \{0 \mid c - \stdMap ^\T  y \in \hat\stdFace _{\ell+1}^*\} $ and 
$\opt{\hat D}$ is attained. In particular, there is $\hat y$ such 
that $c- \stdMap ^\T  \hat y \in  \stdFace _{\ell+1}^*$. 

$(\Leftarrow)$ By $(a)$, $ \opt{\text{\ref{eq:primal_inf}}} = \opt{\hat D} = \{0 \mid c - \stdMap ^\T  y \in \hat\stdFace _{\ell+1}^*\}$, where $\hat\stdFace _{\ell+1} = \stdCone^*\cap \{d_1\}^\perp \cap \cdots \cap \{d_\ell\}^\perp $. Due to $(b)$, $\opt{\text{\ref{eq:primal_inf}}} = \opt{\hat D} = 0 $.
Therefore, \eqref{eq:dual} is not strongly infeasible.

\item Let $\{d_1, \ldots, d_\ell\}$, $\hat {y}$ and $\hat \stdFace _{\ell+1} = \stdCone^*\cap \{d_1\}^\perp \cap \cdots \cap \{d_\ell\}^\perp$ be as 
in item $(i)$. Let $\stdAffine'$ be the affine space $\hat s + \stdSpace '$, 
where  $\stdSpace'$ is spanned by the directions $\{d_1, \ldots, d_\ell\}$  and 
$\hat s = c - \stdMap^\T \hat y$. Since 
$\ell = \distS(\text{\ref{eq:primal_inf}})$, we have $\dimSpace \stdAffine' \leq \distS(\text{\ref{eq:primal_inf}})$.
Suppose for the sake of contradiction that $(\stdAffine ', \stdCone)$ is 
strongly infeasible. Then, we can use   $\{d_1, \ldots , d_\ell \}$ as
reducing directions for $\inf \{ \inProd{\hat s}{x} \mid x \in \stdSpace'^\perp, x \in \stdCone^*\}$.
However, Lemma \ref{lema:str_inf_preserv} implies that $\sup \{0 \mid s  \in  (\hat s + \stdSpace ') \cap \hat \stdFace _{\ell+1}^*\}$ is 
strongly infeasible, which contradicts the fact that $\hat s$ is a feasible solution.

Since the number steps required for Phase 1 of FRA-Poly gives an upper bound for 
$\distS(\text{\ref{eq:primal_inf}})$, we obtain $\distS(\text{\ref{eq:primal_inf}}) \leq \sum _{i=1}^{r} \distP((\stdCone ^i)^*)$.

When \eqref{eq:dual} is weakly infeasible, since $\stdAffine ' \subseteq c - \matRange \stdMap^\T$ and 
$(\stdAffine ', \stdCone)$ is  not strongly infeasible,  it must be the case that 
$(\stdAffine ', \stdCone)$ is weakly infeasible.
\end{enumerate}

\end{proof}

Due to Theorem \ref{theo:dist_comp2}, the bound in Theorem \ref{theo:wi} will usually 
compare favorably to $\ell _{\stdCone^*} - 1$. Moreover, it also recovers 
the bounds described in \cite{lourenco_muramatsu_tsuchiya,lourenco_muramatsu_tsuchiya2}.

\subsection{An application to the intersection of cones}\label{sec:intersection}
In this  subsection, we discuss the case where $\stdCone = \stdCone^1 \cap \stdCone^2$.
We can rewrite \eqref{eq:dual} as a problem over 
$\stdCone^1\times \stdCone^2$ by duplicating the entries. 
\vspace{-2\baselineskip}
\begin{multicols}{2}
	\begin{align} 
	\underset{x^1,x^2}{\inf} & \quad \inProd{c}{x^1+x^2} \label{eq:primal_dup}\tag{$P_{\text{dup}}$}\\ 
	\mbox{subject to} & \quad \stdMap (x^1 + x^2) = 0 \nonumber \\ 
	&\quad (x^1, x^2)\in {\stdCone^1}^*\times {\stdCone^2}^* \nonumber 
	\end{align}
	\break
\begin{align}
\underset{y}{\sup} & \quad \inProd{b}{y} \label{eq:dual_dup} \tag{$D_{\text{dup}}$} \\ 
\mbox{subject to} & \quad (c - \stdMap ^\T y,c - \stdMap ^\T y)  \in \stdCone^1\times \stdCone^2 \nonumber \\ \nonumber 
\end{align} 
\end{multicols}
If we apply FRA-Poly to \eqref{eq:dual_dup}, we will obtain a face $\stdFace^1\times \stdFace ^2$ of 
$\stdCone^1\times \stdCone^2$, so that $\stdFace ^1 \cap \stdFace ^2$ will be a face of 
$\stdCone$ containing $\feasS$. Doing facial reduction using the formulation 
\eqref{eq:dual_dup} might be more convenient, since we need to search for reducing 
directions in $(\stdCone^1)^* \times (\stdCone^2)^*$ instead of $\closure((\stdCone^1)^* + (\stdCone^2))^*$  and deciding 
membership in $(\stdCone^1)^* \times (\stdCone^2)^*$ could be more straightforward than doing the same 
for $\closure((\stdCone^1)^* + (\stdCone^2))^*$. 

Before we proceed we need an auxiliary result. If $\stdCone  = \stdCone^1 \cap \stdCone^2$, 
it is always true that the intersection of a face of $\stdCone^1$ with a 
face of $\stdCone^2$ results in a face of $\stdCone$.
However, it is not  obvious  that every face of $\stdCone$ arises 
as an intersection of faces of $\stdCone ^1$ and $\stdCone ^2$, so we remark 
that as a proposition although it is probably a well-known result.

\begin{proposition}\label{prop:intersection}
Let $\stdFace$ be a nonempty face of $\stdCone ^1 \cap \stdCone ^2$.
Let $\stdFace^1$ and $\stdFace^2$ be the minimal faces of $\stdCone^1$ and $\stdCone^2$, respectively, containing 
$\stdFace$. 
Then $\stdFace = \stdFace^1\cap \stdFace^2$ and $\stdFace^* =  (\stdFace^1)^* + (\stdFace^2)^*$.
\end{proposition}
\begin{proof}
Since $\stdFace ^1 \cap \stdFace ^2$ is a face of $\stdCone^1 \cap \stdCone ^2$,
in order to prove that $\stdFace ^1 \cap \stdFace ^2 = \stdFace$ it is 
enough to show that their relative interiors intersect, which we will do next. By the choice of $\stdFace^1$ and $\stdFace^2$, 
we have $\reInt(\stdFace) \subseteq \reInt (\stdFace^1)$ and 
$\reInt(\stdFace) \subseteq \reInt (\stdFace^2)$, see item $(iii)$ of Proposition 2.2 in \cite{pataki_handbook}\footnote{Proposition 2.2 ensures that $\reInt(\stdFace) $ intersects 
$\reInt (\stdFace^1)$. However, given that $\stdFace _1$ contains $\stdFace$, this is enough for 
the containment $\reInt(\stdFace) \subseteq \reInt (\stdFace^1)$, due to Corollary 6.5.2 in \cite{rockafellar}. The same goes for $\stdFace _2$. }. In particular, 
this implies that $\reInt (\stdFace^1) \cap \reInt (\stdFace^2) \neq \emptyset$.
Therefore, $\reInt (\stdFace^1 \cap  \stdFace^2)  = \reInt (\stdFace^1) \cap \reInt (\stdFace^2) $,
by Theorem 6.5 in \cite{rockafellar}. We conclude that  
$\reInt(\stdFace) \cap \reInt (\stdFace^1 \cap  \stdFace^2) = \reInt(\stdFace) \cap \reInt (\stdFace^1) \cap \reInt (\stdFace^2) \neq \emptyset$. It follows 
that $\stdFace = \stdFace^1 \cap  \stdFace^2$.

Because $\reInt (\stdFace^1) \cap \reInt (\stdFace^2) \neq \emptyset$, the sum $(\stdFace^1)^* + (\stdFace^2)^*$ is closed (see Corollary 16.4.2 in \cite{rockafellar}), so 
that $\stdFace^* = \closure ((\stdFace^1)^* + (\stdFace^2)^*) = (\stdFace^1)^* + (\stdFace^2)^*$.
\end{proof}
%
%
In what follows, we will need two well-known equalities. 
Let $\stdCone ^1, \stdCone ^2$ be closed convex cones and 
$d^1 \in (\stdCone^1)^*, d^2\in (\stdCone ^2)^*$. Then:
\begin{align}
(\stdCone^1 \times \stdCone^2)\cap \{(d^1,d^2) \}^\perp & = \left({\stdCone^1} \cap \{d^1\}^\perp\right) \times \left({\stdCone^2}\cap\{d^2\}^\perp\right) \label{eq:lem1}\\
{\stdCone^1} \cap {\stdCone^2}\cap \{d_1^1\}^\perp \cap \{d_1^2 \}^\perp  & =   {\stdCone^1} \cap {\stdCone^2}\cap \{d_1^1 +d_1^2 \}^\perp. \label{eq:lem2}
\end{align}

\begin{theorem}\label{theo:inter_sing}
Let $\stdCone = \stdCone^1 \cap \stdCone^2$. 
\begin{enumerate}[label=(\roman*)]
	\item Let $\minFace{\eqref{eq:dual_dup}} = \stdFace^1 \times \stdFace^2$ be 
the minimal face of $\stdCone^1 \times \stdCone^2$ containing the set of  feasible 
slacks of \eqref{eq:dual_dup}. Then, $\minFaceD = \stdFace ^1 \cap \stdFace ^2 $.
	\item The singularity degree and the distance to strong duality of  \eqref{eq:dual} satisfy
	\begin{alignat*}{3}
		d(D) \leq & d(D_{\text{dup}}) \leq & 1+ \distP(\stdCone^1) + \distP(\stdCone^2)\\
		\distS(D) \leq & \distS(D_{\text{dup}}) \leq & \distP(\stdCone^1) + \distP(\stdCone^2)
	\end{alignat*}
\end{enumerate}
\end{theorem}
\begin{proof}

$(i)$ $\stdFace^1$ must be the minimal face of $\stdCone^1$ containing 
$\feasS = \{c -\stdMap^\T y \in \stdCone \}$. Otherwise, if some proper face $\widetilde \stdFace$ 
of $\stdFace^1$ is minimal, then $\widetilde \stdFace \times \stdFace^2$ contains 
the feasible slacks of \eqref{eq:dual_dup}, which contradicts the minimality of 
$\hat \stdFace$. The same must hold for $\stdFace ^2$. Then Proposition \ref{prop:intersection} 
implies $\minFaceD = \stdFace ^1 \cap \stdFace ^2$.

$(ii)$ To prove the bounds we first show that given 
$\ell$ reducing  directions for \eqref{eq:dual_dup} we can also construct $\ell$ reducing directions for \eqref{eq:dual} and that there are relations between the faces defined by both sets of directions.

Suppose that $\{(d_1^1,d_1^2), \ldots, (d_\ell^1,d_\ell^2) \}$ are reducing directions for \eqref{eq:dual_dup}. Define
\begin{align}
\stdFace _1^1 \times \stdFace _1^2 & \coloneqq \stdCone^1 \times \stdCone^2 \label{eq:def1}\\
\stdFace _{i+1}^1\times \stdFace _{i+1}^2 & \coloneqq (\stdFace _{i}^1\times \stdFace _{i}^2) \cap \{(d_i^1,d_i^2)\}^\perp \qquad  1 < i \leq \ell. \label{eq:def2}
\end{align}
We will show by induction on $i$
that $\{d_1^1+d_1^2, \ldots, d_\ell^1+d_\ell^2 \}$ are  reducing directions for 
\eqref{eq:dual} and that 
\begin{equation}\stdFace ^1_{i+1} \cap \stdFace ^2 _{i+1} =  \stdCone^1 \cap \stdCone^2 \cap \{d_1^1+d_1^2\}^\perp \cap \cdots \cap \{d_i^1+d_i^2\}^\perp. \label{eq:goal}
\end{equation} 
If \fbox{$i = 1$}, because $(d_1^1,d_1^2)$ is a reducing direction for \eqref{eq:dual_dup}, 
it satisfies 
\begin{equation}
\stdMap (d_1^1 + d_1^2) = 0, \qquad \inProd{c}{d_1^1 + d_1^2} \leq 0, \qquad (d_1^1,d_1^2) \in {\stdCone^1}^* \times {\stdCone^2}^*. \label{eq:base_1}
\end{equation}
Our goal is to show that $d_1^1+d_2^1$ is a reducing direction for \eqref{eq:dual}, that is
\begin{equation}
\stdMap (d_1^1 + d_1^2) = 0, \qquad \inProd{c}{d_1^1 + d_1^2} \leq 0, \qquad d_1^1+d_1^2 \in \closure{({\stdCone^1}^* +{\stdCone^2}^*)} \label{eq:goal_1}
\end{equation}
and that \eqref{eq:goal} holds when $i = 1$. Note that 
\eqref{eq:goal_1} follows directly from \eqref{eq:base_1}. From \eqref{eq:def1}, \eqref{eq:def2} and \eqref{eq:lem1} we have:
$$
\stdFace^1 _2 \times \stdFace^2_2 = ({\stdCone^1} \times {\stdCone^2})\cap\{(d_1^1,d_1^2) \}^\perp  = \left({\stdCone^1} \cap \{d_1^1\}^\perp\right) \times \left({\stdCone^2}\cap\{d_1^2\}^\perp\right).
$$
By \eqref{eq:lem2}, we conclude that 
$\stdFace^1 _2 \cap  \stdFace^2_2 = {\stdCone^1} \cap {\stdCone^2}\cap \{d_1^1\}^\perp \cap \{d_1^2 \}^\perp =  {\stdCone^1} \cap {\stdCone^2}\cap \{d_1^1 +d_1^2 \}^\perp$. 

If \fbox{$i > 1$}, by induction, we have that \eqref{eq:goal} holds up to $i - 1$. By hypothesis, $(d_i^1,d_i^2)$ is a reducing direction, so 
that  
\begin{equation}
\stdMap (d_i ^1 + d_i ^2) = 0, \qquad \inProd{c}{d_i ^1 + d_i ^2} \leq 0 , \qquad (d_i^1,d_i^2) \in {\stdFace _i^1}^* \times {\stdFace _i^2}^*. \label{eq:ind_1}
\end{equation}
We have to show that \eqref{eq:goal} holds and that
\begin{equation}
\stdMap (d_i ^1 + d_i ^2) = 0, \qquad \inProd{c}{d_i ^1 + d_i ^2} \leq 0 , \qquad d_i^1 + d_i^2 \in \closure({\stdFace _i^1}^*+{\stdFace _i^2}^*). \label{eq:ind_2}
\end{equation}
As before, \eqref{eq:ind_2} follows directly from \eqref{eq:ind_1}.
From \eqref{eq:def2}, \eqref{eq:ind_1} and \eqref{eq:lem1} we obtain
$$\stdFace _{i+1}^1 \times \stdFace _{i + 1}^2 = 
(\stdFace^1 _i \cap \{d_i^1 \}^\perp) \times(\stdFace^2 _i \cap \{d_i^2 \}^\perp). $$
We conclude that:
\begin{align*}
\stdFace _{i+1}^1 \cap  \stdFace _{i + 1}^2 & = \stdFace^1  _i \cap \stdFace^2 _i \cap \{d_i^1 + d_i^2 \}^\perp \\
& =  \stdCone^1 \cap \stdCone^2 \cap \{d_1^1+d_1^2\}^\perp \cap \cdots \cap \{d_i^1+d_i^2\}^\perp,
\end{align*}
where the first equality follows from \eqref{eq:lem2} and the second follows from the induction hypothesis. This concludes the induction.


Finally, if $\minFace{\eqref{eq:dual_dup}}  = \stdFace ^1_{\ell+1} \times \stdFace ^2 _{\ell+1}$, 
then $(i)$  implies $\minFaceD =  \stdFace ^1_{\ell+1} \cap \stdFace ^2 _{\ell+1}$.
Similarly, if substituting $\stdCone ^1 \times \stdCone ^2$ in \eqref{eq:dual_dup} for 
$\stdFace ^1_{\ell+1} \times \stdFace ^2_{\ell+1} $ restores strong duality, then the same holds 
for \eqref{eq:dual} if we substitute $\stdCone$ for $\stdFace ^1_{\ell+1} \cap \stdFace ^2 _{\ell+1}$.
This shows that $	d(D) \leq  d(D_{\text{dup}})$ and  $\distS(D) \leq  \distS(D_{\text{dup}})$, respectively. The other bounds  follow from Propositions \ref{prop:phase1} and \ref{prop:min_face}.
\end{proof}

We now consider the case where $\stdCone$ is the doubly nonnegative cone $\doubly{n} = \PSDcone{n} \cap 
\nonNegative{n}$, where $\nonNegative{n}$ is the cone of $n\times n$ symmetric matrices 
with nonnegative entries. This cone is important because it can be used as a relatively 
tractable relaxation for the cone of completely positive matrices, see \cite{Yoshise10,KKT15,AKKT2014}. 
\begin{corollary}\label{col:doubly_sing}
When $\stdCone = \doubly{n}$, we have $d(D) \leq n$ and $\distS(D) \leq n-1$.
\end{corollary}
\begin{proof}
Follows from Theorem \ref{theo:inter_sing} since $\distP(\PSDcone{n}) = n-1$ and $\distP(\nonNegative{n}) = 0$.
\end{proof}

We will compare the bound in Corollary \ref{col:doubly_sing} with the 
one predicted by the classical FRA. To do that, we need to compute $\ell _{\doubly{n}}$.
\begin{proposition}\label{prop:doubly_chain}
The longest chain of nonempty faces in $\doubly{n}$ has length 
$\frac{n(n+1)}{2} + 1$, which is the maximum possible for a cone 
contained in $\S^n$.
\end{proposition}
\begin{proof}
Maximality follows from the fact  that the dimension of $\S^n$ is
$\frac{n(n+1)}{2}$ and that if we have 
two faces such that $\stdFace \subsetneq \hat \stdFace$ then  
$\dimSpace(\stdFace) < \dimSpace(\hat \stdFace)$. 

Let $\mathcal{G}$ be any set of tuples $(i,j)$ with $i,j \in \{1,\ldots, n\}$  and 
let $\nonNegative{n}(\mathcal{G})$ be the face of $\nonNegative{n}$ which 
corresponds to the matrices $x$ such that the only entries $x_{i,j}$ that are allowed 
to be nonzero are the ones for which either $(i,j) \in \mathcal{G}$ or 
$(j,i) \in \mathcal{G}$. We will first define two chains of faces of $\nonNegative{n}$.
First, let  $\mathcal{G}_0 = \emptyset$ and define $\mathcal{G}_i = \mathcal{G}_{i-1} \cup \{(i,i) \}$ for 
$i \in \{1, \ldots, n\}$.  We now consider 
the following construction written in pseudocode.
\begin{enumerate}[label=\ ]
	\item $k \leftarrow 1$, $\mathcal{H}_0 \leftarrow \mathcal{G}_n$
	\item \textbf{For} $i \leftarrow 1, i \leq n$ \textbf{do}
	\begin{enumerate}[label=\ ]
		\item \textbf{For} $j \leftarrow 1$, $j < i$ \textbf{do}
		\begin{enumerate}[label=\ ]
			\item \hspace{0.5cm}$\mathcal{H}_k \leftarrow \mathcal{H}_{k-1}  \cup \{i,j\}  $
			\item \hspace{0.5cm}$k \leftarrow k + 1$
			\item \hspace{0.5cm}$j \leftarrow j + 1$.
		\end{enumerate}
		\item $i \leftarrow i + 1$.
	\end{enumerate}
\end{enumerate}
The idea is to add one non-diagonal entry per iteration, so 
that $\nonNegative{n}(\mathcal{H}_k) \subsetneq    \nonNegative{n}(\mathcal{H}_{k+1})$.
First $(2,1)$ will be added, then $(3,1),(3,2)$ and so on.
We have
\begin{equation*}
\PSDcone{n} \cap \nonNegative{n}(\mathcal{G}_0) \subsetneq \ldots 
\subsetneq \PSDcone{n} \cap \nonNegative{n}(\mathcal{G}_n) 
\subsetneq \PSDcone{n} \cap \nonNegative{n}(\mathcal{H}_1) \subsetneq \ldots 
\subsetneq \PSDcone{n} \cap \nonNegative{n}(\mathcal{H}_{\frac{n(n-1)}{2}}) 
\end{equation*}
and all inclusions are indeed strict. The first $n$ inclusions are strict 
because  $ \PSDcone{n} \cap \nonNegative{n}(\mathcal{G}_i) = \nonNegative{n}(\mathcal{G}_i)$ and 
it is clear that $\nonNegative{n}(\mathcal{G}_i) \subsetneq \nonNegative{n}(\mathcal{G}_{i+1})$.
Now, let $\mathcal{I}_n$ denote the $n\times n$ identity matrix.
If $k > 0$ and  $x \in \reInt \nonNegative{n}(\mathcal{H}_k)$ then $x_{i,j} > 0$ for 
some $(i,j)$ entry  such that neither $(i,j)$ nor $(j,i)$ belong to $\mathcal{H}_{k-1} $.
For $\alpha > 0$ sufficiently large, we have $x + \alpha \mathcal{I}_n \in \PSDcone{n} \cap \nonNegative{n}(\mathcal{H}_k) $
and $x + \alpha \mathcal{I}_n \not \in \PSDcone{n} \cap \nonNegative{n}(\mathcal{H}_{k-1}) $. This shows the 
remainder of the containments and concludes the proof, since the chain 
has length $\frac{n(n+1)}{2} + 1$.
\end{proof}

For feasible problems, the classical FRA analysis gives either the bound 
$\ell_{\doubly{n}} - 1 = \frac{n(n+1)}{2} $ or, using Theorem \ref{theo:inter_sing}, 
the bound $\ell_{\PSDcone{n}} - 1 +  \ell_{\nonNegative{n}} - 1 = n + \frac{n(n+1)}{2}$.
Both bounds are quadratic in $n$ in opposition to the linear bound obtained 
in Corollary \ref{col:doubly_sing}.

\small{
	\section*{Acknowledgements}
We thank the referees and the associate editor for the detailed feedback we were given. The article benefited greatly from it and, in particular, 
the 2nd referee motivated the discussion in Section \ref{sec:tight} and suggested the first instance in Appendix \ref{app:ex}, for which we are grateful.
}
\bibliographystyle{plainurl}
\bibliography{bib}
\appendix
\section{Partial Polyhedrality and Slater's condition}\label{app:partial}
Let $f:\Re^n \to \Re\cup \{-\infty,+\infty \}$ be a convex function.
We denote the domain of $f$ by $\fDom f = \{x \in \Re^n \mid f(x) < \infty \}$. If $\fDom f \neq \emptyset $ and $f$ is never $-\infty$, then $f$ is said to be proper.
Its conjugate will be denoted by $f^*$ and it satisfies $f^*(s) = \sup _{x} \inProd{s}{x} - f(x)$.
If the epigraph of $f$ is a polyhedral set, then $f$ is said to be a polyhedral function. 
We recall Theorem 20.1 from \cite{rockafellar}.

\begin{theorem}[Rockafellar]\label{theo:conv}
Let $f_1, \ldots, f_m$ be proper convex functions and let $f_{k+1}, \ldots ,$ $f_m$ be 
polyhedral functions. Suppose also that
$$
\reInt (\fDom f_1) \cap \cdots \cap \reInt (\fDom f_k) \cap \fDom f_{k+1} \cap \cdots \cap \fDom f_m \neq \emptyset.
$$
Then the following holds:
\begin{equation*}
(f_1 + \cdots + f_m)^*(s) = \inf \{ f_1^*(s_1) + \cdots + f_m^*(s_m) \mid s_1 + \cdots + s_m = s \},
\end{equation*}
where for each $s$ the infimum is attained whenever it is finite.
\end{theorem}

\begin{proposition}
Let $\stdCone = \stdCone ^1\times \stdCone ^2$, where  $\stdCone ^1\subseteq \Re^{n_1},\stdCone ^2 \subseteq \Re^{n_2}$ are closed convex cones such that $\stdCone ^2$ is polyhedral.
\begin{enumerate}[label=({\it\roman*})]
\item If $\pOpt$ is finite and \eqref{eq:primal} satisfies the PPS condition, then $\pOpt = \dOpt$ and the dual optimal value is attained.
\item If $\dOpt$ is finite and \eqref{eq:dual} satisfies the PPS condition, then $\pOpt = \dOpt$ and the primal 
optimal value is attained.
\end{enumerate}
\end{proposition}
\begin{proof}
We will prove $(i)$ first. Let 
$f_1$ be such that $f_1(x) = \inProd{c}{x}$ if $Ax = b$ and $+\infty$ otherwise.
Let  $f_2$ be the indicator function of $\Re^{n_1}\times (\stdCone ^2)^*$ and $f_3$ be the indicator 
function of $(\stdCone ^1)^*\times \Re^{n_2}$. Since there is  a primal feasible solution $x = (x_1,x_2)$ such 
that $x_1 \in \reInt (\stdCone ^1)^*$, we have that $\fDom f_1 \cap \fDom f_2 \cap \reInt (\fDom f_3)$ is 
nonempty. In addition, $f_1$ and $f_2$ are polyhedral functions. Let us now observe that:
\begin{align*}
f_1^*(s) = \left\{
  \begin{array}{l l}
    \inProd{b}{y}  & \quad \text{if there is $y$ with $s-c = \stdMap^\T y$}\\
    +\infty & \quad \text{otherwise}
  \end{array}\right. 
\end{align*}
Note that, due to feasibility, for fixed $s$, $\inProd{b}{y}$ does not depend on the choice 
of $y$, as long as $c+\stdMap ^\T y = s$. This is because since there is $x$ such that $Ax = b$, we have $\inProd{b}{y} = \inProd{x}{s-c}$.
The conjugate $f_2^*$ is the indicator function of $- \{0\}\times \stdCone ^2$ and 
$f_3^*$ is the indicator function of $- \stdCone ^1\times \{0\}$. 
Applying Theorem \ref{theo:conv} with $s = 0$, we have:
\begin{align*}
(f_1+f_2+f_3)^*(0) & = \inf \left\{ \inProd{b}{y} \mid  c+ \stdMap ^\T y = s_1, s_1 - (s_3,s_2) = 0,  s_2 \in  \stdCone ^2, s_3 \in \stdCone ^1      \right\} \\
 & = \inf \left\{ \inProd{b}{y} \mid  c+ \stdMap ^\T y = s_1, s_1 \in  \stdCone ^1\times \stdCone ^2      \right\} \\
& = - \sup \left\{ \inProd{b}{y} \mid  c -\stdMap ^\T y = s_1,  s_1 \in  \stdCone ^1 \times \stdCone ^2      \right\},
\end{align*}
where the $\sup$ in the last equation is attained. So, there is some dual feasible $y$ such that 
$(f_1+f_2+f_3)^*(0) =  \inProd{b}{y}$. However, using the definition of conjugate, we also have:
\begin{align*}
(f_1+f_2+f_3)^*(0) = - \inf \{ \inProd{c}{x} \mid Ax = b, x \in \stdCone ^1\times \stdCone ^2\} = -\theta _P.
\end{align*}
It follows that $\pOpt = \dOpt$ and the dual is attained at $y$. 
To prove $(ii)$, let $g_1 = f_1^*$, and let $g_2$ and $g_3$ be the indicator 
functions  of $\Re^{n_1}\times \stdCone ^2$ and  $\stdCone ^1 \times \Re^{n_2}$, respectively. Again, 
it is enough to compute $(g_1+g_2+g_3)^*(0)$ using both the definition of conjugate function and 
using Theorem \ref{theo:conv}. 
\end{proof}

\section{Examples} \label{app:ex}
%

 

\begin{example}
We will apply FRA-Poly to the following problem.
\begin{align}
{\text{find}} & \quad y \nonumber\\ 
\mbox{subject to} & \quad (y_1,-y_1) \times \begin{pmatrix}y_1 & y_2 \\ y_2 & y_3 \end{pmatrix} \in \Re^2_ + \times \PSDcone{2}. \nonumber
\end{align}
At the first step we have $\stdFace _1 = \Re^2_ + \times \PSDcone{2}$ and we build \eqref{eq:red} and \eqref{eq:red_dual} using $e = 0\times I_2$ and $e^* = (1,1) \times I_2$, where 
$I_2$ is the $2\times 2$ identity matrix. Solving \eqref{eq:red} and \eqref{eq:red_dual}, suppose that we have found the reducing direction 
$$
d_1 = (1,2) \times \begin{pmatrix}1 & 0 \\ 0 & 0 \end{pmatrix}.
$$
Then $\stdFace _2 = \stdCone \cap \{d_1\}^\perp = \{0\} \times \hat \stdFace$, where the matrices in $\hat \stdFace $ have zeros in all entries 
except in the $(2,2)$ entry. Note that $\stdFace _2$ is polyhedral, so when 
it is time to build  $\eqref{eq:red}$ and \eqref{eq:red_dual} again, we will take $e = 0$ and since the {\PPS} condition is satisfied, 
we will have $\opt{\text{\ref{eq:red}}} > 0$.
Note that this is a case where Phase 1  ends  at the minimal face and 
there is no need to proceed further, although we might want to solve the problem in Phase 2 in order to obtain a point in 
$\reInt \minFaceD$.

Note that if the first block of $d_1$ were $(0,1)$ instead of $(1,2)$, then we would have 
$\stdFace _2 =   \Re_+ \times \{0\} \times \hat \stdFace$, so we would need a Phase 2 iteration  to 
find a reducing direction $d_2$ such that $\minFaceD = \stdFace _2 \cap \{d_2\}^\perp$. Still, if we are able to implement 
the choice in \eqref{eq:dir_ri} we will never need to go for a second direction, since we will always take the most 
interior direction possible.
\end{example}

%
%
\medskip
Let 
$\stdCone = \SOC{t_1} \times \ldots \times \SOC{t_{r_1}} \times \PSDcone{n_1}  \times \ldots \times \PSDcone{n_{r_2}}$.
We will assume that $r_1 + r_2 > 0$,
$t_j \geq 3$ and $n_j \geq 3$ for every $j$. Given $x$, we will use 
$x^j_{i,k}$ to denote the $(i,k)$ entry of the $j$-th matrix block and $x^j_i$ to 
denote the $i$-entry of the $j$-th  vector block. We will also 
use the same notation to single out a few special elements.
For $j \in [1,r_2 ], a^j_{i,k} \in \stdCone$ is
such that all its blocks are zero except for the block corresponding to $\PSDcone{n_j}$. In that block 
$a^j_{i,k}$ contains the $n_j\times n_j$ matrix that has one at the 
$(i,k)$ and $(k,i)$  entries and zero elsewhere. Similarly, for $j \in [1, r_1]$, $a^j_i \in \stdCone$ is such 
that all its blocks are zero except for the block corresponding to $\SOC{t_j}$, where $a^j_i$ corresponds to the $i$-th unit vector.  

Recall that if $d = (d_0,\overline{d})$ and $x = (x_0,\overline{x})$ are points 
of $\SOC{n}$ with $d_0,x_0 \in \Re$ and $\overline{d},\overline{x} \in \Re^{n-1}$, then 
$\inProd{d}{x} = 0$ implies  $d_0\overline{x} + x_0\overline{d} = 0 $. In particular, if 
$d$ is a nonzero boundary point of $\SOC{n}$, then the face $\SOC{n} \cap \{d\}^\perp$ is 
equal to the half-line $h_{d'} = \{\alpha d' \mid \alpha \geq 0 \}$ where 
$d' = (d_0,-\overline{d})$. We also have $h_{d'}^* = \{x \mid \inProd{x}{d'} \geq 0 \} $.

Let $\stdSpace ^\perp$ be the space spanned by the following vectors:
\begin{enumerate}
	\item $a^1_{1} + a^1_{2}$ and $\{a^{j-1}_3 + a^j_{1} + a^j_{2} \mid 1 < j \leq r_1\}$,
	\item $a^{r_1}_3 + a^1_{1,1}$ and  $\{a^{1}_{i,i} + a^{1}_{i-1,i+1}  \mid 1 < i < n_1\}$ (if $r_1 = 0$, use $a^1_{1,1}$ instead of  $a^{r_1}_3 + a^1_{1,1}$),
	\item $a^{j-1}_{n_j,n_{j-1}} + a^j_{1,1}$ and  $\{a^{j}_{i,i} + a^{j}_{i-1,i+1}  \mid 1 < i < n_j\}$, for $1 < j \leq r_2$.
\end{enumerate}

\begin{remark}
It will be helpful  to keep in mind the case where $r_1 = r_2 = 2$, $n_1 = n_2 = t_1 = t_2 = 3$.
In this case, $\stdSpace ^\perp$ is spanned by elements having the following format
	\begin{equation*}
	\begin{pmatrix}y_1 \\ y_1 \\ y_2 \end{pmatrix} \times \begin{pmatrix}y_2 \\ y_2 \\ y_3 \end{pmatrix} \times 
	\begin{pmatrix}y_3 & 0 & y_4 \\ 0 & y_4 & y_5 \\ y_4 & y_5 & 0  \end{pmatrix} \times 
	\begin{pmatrix}y_5 & 0 & y_6 \\ 0 & y_6 & 0 \\ y_6 & 0 & 0   \end{pmatrix}. 
	\end{equation*} 
\end{remark}

\begin{proposition}\label{prop:worst_bound}
Consider the problem \eqref{eq:dual} with $c = 0$ and $\stdMap$ such that $ \matRange \stdMap^\T = \stdSpace$, where $\stdSpace ^\perp$ is 	the subspace constructed above. Then 
	$
	d(D) =	 r_1 + \sum _{j = 1}^{r_2} (n_j  - 1).
	$
\end{proposition}
\begin{proof}
	First, suppose that $r_1 > 0$.
	The first reducing direction must be some  $x \in (\stdCone \cap \stdSpace^\perp )\setminus \stdCone^\perp$. However, if $x \in \stdSpace^\perp$,
	then 	$x^1_1 = x^1_2$. Then, because $x^1 \in \SOC{t_1}$, we have 
	$x^1_i = 0$ for all $i\geq 3$.  Therefore, the coefficient of $a^{1}_3 + a^2_{1} + a^2_{2}$ appearing 
	in $x$ must be zero as well. It follows that all  blocks of $x$ are zero, 
	except for $x^1$. We conclude that $x$ must be a positive multiple of $ a^1_{1} + a^1_{2}$. So let $d_1 = a^1_{1} + a^1_{2}$, we then have 
	$$\stdFace _2 = \stdCone \cap \{d_1\}^\perp = h_{d_1'} \times \SOC{t_{2}} \times \ldots \times \SOC{t_{r_1}} \times \PSDcone{n_1}  \times \ldots \times \PSDcone{n_{r_2}}, $$
	where $h_{d_1'}$ is contained in $\SOC{t_1}$ and is the half-line along the direction defined by the nonzero part of $a^1_{1} - a^1_{2}$.
At the next step, it turns out that only the positive multiples 
	of $a^{1}_3 + a^2_{1} + a^2_{2}$ belong to $(\stdFace _2^* \cap \stdSpace^\perp) \setminus \stdFace _2^\perp $. This means 
	that facial reduction must proceed by successively selecting  positive multiples of:
	\begin{enumerate}
		\item $d_1 = a^1_{1} + a^1_{2}$ and $d_j = a^{j-1}_3 + a^j_{1} + a^j_{2}$, for $1 < j \leq r_1$.
	\end{enumerate}
	After $r_1$ steps, all the Lorentz cone blocks will be transformed to half-lines 
	and we will have $\stdFace _{r_1+1} =  h_{d_1'} \times \ldots \times h_{d_{r_1}'} \times \PSDcone{n_1}  \times \ldots \times \PSDcone{n_{r_2}} $, 
	where for every $1 < j \leq r_1 $, $h_{d_j'}$ is the half-line in $\SOC{t_j}$ along the direction defined by the nonzero part of $a^j_{1} - a^j_{2}$. 	If $r_2 = 0$, we are done.
	Otherwise, we  have $(\stdFace _{r_1+1}^*\cap \stdSpace ^\perp)\setminus \stdFace _{r_1+1}^\perp  = \{t(a^{r_1}_3 + a^1_{1,1}) \mid t > 0 \}$.
	Again, we must proceed ``one row at time'' and select positive multiples of 
	$a^{1}_{i,i} + a^{1}_{i-1,i+1}$ for  $1 < i < n_1$ as the reducing directions. In total, we find $n_1 - 1$ directions before we can move to the next block. For each block 
	$n_j -1$ directions will be found, so in total we obtain $r_1 + \sum _{j = 1}^{r_2} (n_j  - 1)$
	directions. 	The case $r_1 = 0$ follows similarly.
\end{proof}
\end{document}